\documentclass[11pt, reqno]{amsart}

\usepackage{amsmath, amsthm, amscd, amsfonts, amssymb, graphicx, color}
\usepackage[bookmarksnumbered, colorlinks, plainpages]{hyperref}
\usepackage{dsfont,xcolor}
\usepackage{mathrsfs}
 \makeatletter
\let\reftagform@=\tagform@
\def\tagform@#1{\maketag@@@{(\ignorespaces\textcolor{blue}{#1}\unskip\@@italiccorr)}}
\renewcommand{\eqref}[1]{\textup{\reftagform@{\ref{#1}}}}
\makeatother
\hypersetup{colorlinks=true, linkcolor=red, anchorcolor=green,
citecolor=cyan, urlcolor=red, filecolor=magenta, pdftoolbar=true}

\newtheorem{definition}{Definition}[section]
\newtheorem{proposition}{Proposition}[section]
\newtheorem{theorem}{Theorem}[section]
\newtheorem{remark}{Remark}[section]
\newtheorem{lemma}{Lemma}[section]
\newtheorem{corollary}{Corollary}[section]
\newcommand{\Tr}{\mathrm{Tr}}           

\newcommand{\SR}{\mathcal{SR}}
\newcommand{\FSN}{\mathcal{FSN}}
\newcommand{\cone}{\mathrm{cone}}
\newcommand{\conv}{\mathrm{conv}}

\begin{document}

\title[Fractional    $k$-positive maps]{Fractional $k$-positivity: a continuous refinement of the $k$-positive scale}

\author{Mohsen Kian}
\address{$^1$Department of Mathematics, University of Bojnord, P. O. Box 1339, Bojnord 94531, Iran}
\email{kian@ub.ac.ir }


\begin{abstract}
We introduce a real-parameter refinement of the classical integer hierarchies underlying Schmidt number, block-positivity, and $k$-positivity for maps between matrix algebras.
Starting from a compact family of $\alpha$-admissible unit vectors ($\alpha\in[1,d]$), we define closed cones $\mathsf K_\alpha$ of bipartite positive operators that interpolate strictly between successive Schmidt-number cones, together with their dual witness cones.  Via the Choi--Jamio\l{}kowski correspondence this yields a matching filtration of map cones $\mathsf P_\alpha$, recovering the usual $k$-positive/$k$-superpositive classes at integer parameters and complete positivity at the top endpoint.

Two results show that the fractional levels capture genuinely new structure.  First, we prove a \emph{fractional Kraus theorem}: $\alpha$-superpositive maps are precisely the completely positive maps admitting a Kraus decomposition whose Kraus operators satisfy an explicit singular-value (Ky--Fan) constraint, extending the classical rank-$k$ characterization.  Second, for non-integer $\alpha$ the cones $\mathsf P_\alpha$ fail stability under CP post-composition, highlighting a sharp structural transition away from the integer theory.  Finally, we derive sharp thresholds on canonical symmetric families (including the depolarizing ray and the isotropic slice), turning familiar stepwise criteria into continuous, computable profiles.
\end{abstract}

\maketitle

\section{Introduction}

Positive linear maps between matrix algebras and the convex cones they form are central objects
in operator theory, matrix analysis, and (via the Choi--Jamio{\l}kowski correspondence) quantum information theory.
On the one hand, positivity constraints such as $k$-positivity and complete positivity encode
noncommutative order structure and dilation phenomena; on the other, positive but not completely positive maps
serve as entanglement witnesses through the separability criteria of Peres--Horodecki and Horodecki \cite{Peres96,HHH96}.
A unifying perspective is to view these classes as closed convex cones in mapping spaces and to study
their duality, facial structure, and stability under natural operations (composition, amplification, local conjugations);
see, e.g., St{\o}rmer's work on mapping cones \cite{StormerMC,StormerBook} and more recent systematic treatments
of cones in mapping spaces \cite{SkowronekStormerZyczkowski,GiardKyeStormer}.

A classical integer hierarchy runs through much of the subject.  Fix $d=\min\{n,m\}$ and consider bipartite vectors
$\psi\in\mathds{C}^n\otimes\mathds{C}^m$ with Schmidt rank at most $k$.
Their rank-one projectors generate the cone of $k$-separable (or $k$-entangled) positive operators, and the dual cone
consists of $k$-block-positive operators (entanglement witnesses of Schmidt class $k$); see also \cite{KianEntCone}.
Via the Choi--Jamio{\l}kowski identification \cite{Jamiolkowski72,Choi75,WatrousTQI},
these operator cones correspond to the cones of $k$-superpositive maps (completely positive maps admitting Kraus operators
of rank $\le k$) and $k$-positive maps, respectively; see \cite{SkowronekStormerZyczkowski}.
At the level of states, the Schmidt number introduced by Terhal--Horodecki \cite{TerhalHorodeckiSN} yields a stepwise
classification: membership in the $k$-separable cone changes at discrete thresholds (e.g.\ on isotropic/Werner slices).

\medskip
\noindent
\textbf{Aim and main idea.}
The integer parameter $k$ produces genuine geometric and operational discontinuities:
many natural one-parameter families exhibit \emph{step} membership rules in $k$ with sharp breakpoints.
A guiding goal of this paper is to refine the integer hierarchy by introducing a \emph{global real parameter}
$\alpha\in[1,d]$ that interpolates between $k$ and $k+1$ in a controlled, intrinsically convex-geometric way.
Rather than modifying the ambient space, we refine the \emph{admissible rank-one generators}:
between $k$ and $k+1$ we allow one additional Schmidt coefficient, but constrain its size by an explicit ratio bound.
This produces a continuously varying family of compact generating sets and hence a nested family of closed cones.

\medskip
\noindent\textbf{Contributions and organization.}
Fix $n,m\in\mathbb N$ and set $d:=\min\{n,m\}$. We introduce a fractional interpolation between the integer
hierarchies (Schmidt number/$k$-block positivity and $k$-positivity) by testing positivity only on a family
$\mathcal V_\alpha$ of \emph{$\alpha$-admissible} unit vectors, $\alpha\in[1,d]$, and taking the resulting closed conic hull.
This yields nested cones of states $\{\mathsf K_\alpha\}_{\alpha\in[1,d]}$ and of maps $\{\mathsf P_\alpha\}_{\alpha\in[1,d]}$
that agree with the classical levels at integers and interpolate strictly between them.

While the cones $\mathsf K_\alpha$ and $\mathsf P_\alpha$ recover the classical integer hierarchies at $\alpha=k\in\{1,\dots,d\}$,
their behavior for non-integer $\alpha$ is not a formal convex interpolation.
In particular, away from integers the map cone $\mathsf P_\alpha$ fails a natural stability property under CP post-composition
(Proposition~\ref{prop:failure-cp-new}), which sharply distinguishes the fractional levels from the usual mapping-cone structures
appearing in the theory of $k$-positivity.

\smallskip
\noindent Our main results are:
\begin{itemize}
\item \emph{Geometry and interpolation of the fractional cones.}
We establish compactness of the generating sets and closedness/duality of the induced cones,
prove strict inclusions between consecutive integer levels, and show recovery of the classical integer hierarchies
(Propositions~\ref{prop:compact-lambda}, \ref{prop:strict-inclusions}, \ref{prop:integer-recovery}).

\item \emph{Fractional Kraus theorem.}
We characterize $\alpha$-superpositive maps as exactly those completely positive maps that admit a Kraus decomposition
with Kraus operators satisfying the same fractional singular-value constraint underlying $\mathcal V_\alpha$
(Theorem~\ref{thm:fractional-kraus}).

\item \emph{Sharp thresholds and closed-form profiles.}
We compute the exact $\alpha$-positivity threshold for the depolarizing family
$\Phi_t(X)=\Tr(X)I-tX$ (Theorem~\ref{thm:alpha-threshold-Phi-t}),
and derive a sharp fractional profile on the isotropic slice $\mathrm{span}\{I,P_\omega\}$
(Theorem~\ref{thm:flagship-slice}), including a closed-form inversion for the fractional Schmidt index of isotropic states
(Corollary~\ref{cor:FSN-rhoF}).
\end{itemize}

\smallskip
\noindent The paper is organized as follows.
Section~\ref{sec:preliminaries} fixes notation.
Section~\ref{sec:alpha-positivity} develops the fractional cones $\mathsf K_\alpha,\mathsf{BP}_\alpha$ and the induced map cones
$\mathsf P_\alpha,\mathsf{SP}_\alpha$, and proves their structural properties including the fractional Kraus theorem.
Section~\ref{sec:applications} contains the sharp depolarizing threshold and the isotropic-slice computations.

\section{Preliminaries}
\label{sec:preliminaries}

This section fixes notation and recalls the basic objects used throughout the paper.
We work with finite-dimensional matrix algebras, the vectorization/matricization correspondence, and the Choi
representation of linear maps. We also briefly recall Schmidt rank/Schmidt number terminology, since our fractional
constructions interpolate the familiar integer hierarchies.

For $n\in\mathbb{N}$, let $\mathbb{M}_n:=\mathbb{M}_n(\mathds{C})$ be the algebra of $n\times n$ complex matrices and $\mathbb{H}_n$ be the real vector space of Hermitian matrices.
The (closed, convex) cone of positive semidefinite matrices is denoted by $\mathbb{M}_n^{+}$.
 We use $\Tr(\cdot)$ for the usual trace and $\langle X,Y\rangle:=\Tr(X^\ast Y)$ for the Hilbert--Schmidt inner product on $\mathbb{M}_n$.

For $k,n\in\mathds{N}$ we identify $\mathbb{M}_k\otimes \mathbb{M}_n \cong \mathbb{M}_{kn}$ via the Kronecker product, and also with
block matrices $[X_{ij}]_{i,j=1}^k$ with blocks $X_{ij}\in \mathbb{M}_n$.

Let $\mathcal{H}_A\simeq \mathds{C}^n$ and $\mathcal{H}_B\simeq \mathds{C}^m$.
Every vector $\psi\in \mathcal{H}_A\otimes \mathcal{H}_B$ admits a \emph{Schmidt decomposition}
\[
\psi=\sum_{i=1}^{r} s_i\, x_i\otimes y_i,
\]
where $r\le \min\{n,m\}$, $\{x_i\}_{i=1}^r\subset \mathcal{H}_A$ and $\{y_i\}_{i=1}^r\subset \mathcal{H}_B$
are orthonormal families, and the \emph{Schmidt coefficients} satisfy
\[
s_1\ge s_2\ge \cdots \ge s_r>0.
\]
If $\|\psi\|=1$, then $\sum_{i=1}^r s_i^2=1$.

The \emph{Schmidt rank} of a nonzero vector $\psi$, denoted $\SR(\psi)$, is the number $r$ of nonzero
Schmidt coefficients in any Schmidt decomposition of $\psi$.

\smallskip
\noindent
Next we fix the vectorization convention that allows us to pass freely between bipartite vectors and rectangular matrices.

Fix orthonormal bases $\{e_i\}_{i=1}^n$ of $\mathds{C}^n$ and $\{f_j\}_{j=1}^m$ of $\mathds{C}^m$.
For $\psi\in\mathds{C}^n\otimes\mathds{C}^m$, write
\[
\psi=\sum_{i=1}^n\sum_{j=1}^m a_{ij}\,e_i\otimes f_j,
\]
and define its \emph{matricization} $\mathrm{mat}(\psi)\in \mathbb{M}_{n,m}$ by
\[
\mathrm{mat}(\psi):=[a_{ij}]_{i,j}.
\]
Then the correspondence $\psi\mapsto \mathrm{mat}(\psi)$ is a linear isometry between
$\mathds{C}^n\otimes\mathds{C}^m$ (with the Hilbert space norm) and $\mathbb{M}_{n,m}$ (with the Frobenius norm), i.e.
\begin{align}\label{Frobs}
\|\psi\|^2=\|\mathrm{mat}(\psi)\|_F^2,\qquad
\|\psi-\varphi\|^2=\|\mathrm{mat}(\psi)-\mathrm{mat}(\varphi)\|_F^2.
\end{align}
Moreover,
\[
\SR(\psi)=\mathrm{rank}(\mathrm{mat}(\psi)),
\]
and if $\psi$ has Schmidt coefficients $s_1(\psi)\ge\cdots\ge s_d(\psi)\ge 0$
(extended by zeros if necessary), then
\[
s_i(\psi)=\sigma_i(\mathrm{mat}(\psi)),\qquad i=1,\dots,d,
\]
where $\sigma_1\ge \sigma_2\ge \cdots$ denote the singular values.

For a matrix \(X\), we write \(\sigma_1(X)\ge \sigma_2(X)\ge\cdots\) for its singular values,
\(\|X\|_1:=\sum_j \sigma_j(X)\) for the trace norm, and
\(\|X\|_{(k)}:=\sum_{j=1}^k \sigma_j(X)\) for the Ky--Fan \(k\)-norm.

We  use the standard column-stacking vectorization $\mathrm{vec}:\mathbb{M}_d\to\mathds{C}^{d}\otimes\mathds{C}^{d}$ and its inverse
$\mathrm{mat}$, characterized by the identity
\begin{equation}
\label{eq:vec-AXB}
\mathrm{vec}(A X B)=\big(B^{\mathsf T}\otimes A\big)\mathrm{vec}(X),
\qquad A,B,X\in M_d.
\end{equation}

\begin{definition}
We say that $W\in\mathbb{M}_n\otimes \mathbb{M}_m $ is \emph{block-positive} if
\[
\langle x\otimes y,\ W(x\otimes y)\rangle \ge 0
\quad\text{for all } x\in\mathds{C}^n,\ y\in\mathds{C}^m.
\]
\end{definition}

\begin{definition}
Fix $k\in\mathbb{N}$. We say that $W\in\mathbb{M}_n\otimes \mathbb{M}_m$ is \emph{$k$-block-positive} if
\[
\langle \psi,\ W\psi\rangle \ge 0
\quad\text{for all unit vectors } \psi\in\mathds{C}^n\otimes\mathds{C}^m \text{ with } \SR(\psi)\le k.
\]
\end{definition}

A linear map $\Phi:\mathbb{M}_n\to \mathbb{M}_m$ is \emph{Hermitian-preserving} if $\Phi( \mathbb{H}_n)\subseteq \mathbb{H}_m$.
It is \emph{positive} if $\Phi(\mathbb{M}_n^+)\subseteq \mathbb{M}_m^+$.

For $k\in\mathbb{N}$, the \emph{$k$-amplification} of $\Phi$ is $\mathrm{id}_k\otimes \Phi: \mathbb{M}_k\otimes \mathbb{M}_n\to \mathbb{M}_k\otimes \mathbb{M}_m$.
The map $\Phi$ is  $k$-positive  if $\mathrm{id}_k\otimes \Phi$ is positive. If $\Phi$ is $k$-positive for all $k\in\mathbb{N}$, then it is called completely positive (CP).

\smallskip
\noindent
We will frequently identify linear maps with matrices via the Choi representation, so we record the convention we use.
 Let $\{E_{ij}\}_{i,j=1}^n$ be the standard matrix units in $\mathbb{M}_n$.
The \emph{Choi matrix} of a linear map $\Phi:\mathbb{M}_n\to \mathbb{M}_m$ is
\[
C_\Phi \;:=\;\sum_{i,j=1}^n E_{ij}\otimes \Phi(E_{ij})\ \in\ \mathbb{M}_n\otimes \mathbb{M}_m.
\]
If $\Phi$ is Hermitian-preserving then $C_\Phi$ is Hermitian.

Equivalently, if $|\Omega\rangle:=\sum_{i=1}^n e_i\otimes e_i\in \mathds{C}^n\otimes\mathds{C}^n$ is the (unnormalized)
maximally entangled vector, then
\[
C_\Phi = (\mathrm{id}_n\otimes \Phi)\big(|\Omega\rangle\langle \Omega|\big).
\]
We identify linear maps with their Choi matrices. The associated Choi pairing is
\[
\langle \Phi,\Psi\rangle_{\mathrm{Choi}}
:= \operatorname{Tr}(C_\Phi\,C_\Psi),
\]
i.e., the Hilbert--Schmidt inner product of Choi matrices.

There exists a standard correspondence as
\[
\Phi\ \text{is a  $k$-positive map}
\quad\Longleftrightarrow\quad
C_\Phi\ \text{is a $k$-block-positive matrix}.
\]



\section{\texorpdfstring{Fractional $\alpha$-positivity}{Fractional alpha-positivity}}
\label{sec:alpha-positivity}
\noindent
We now introduce the central objects of the paper.
The guiding idea is to interpolate the integer Schmidt-rank hierarchy by restricting positivity tests to a controlled family of vectors.

Let $d:=\min\{n,m\}$, so every $\psi\in\mathds C^n\otimes\mathds C^m$ has at most $d$ nonzero Schmidt coefficients.

Given $\alpha\in[1,d]$, set $k:=\lfloor\alpha\rfloor$, $\theta:=\alpha-k\in[0,1)$, and $r:=\lceil\alpha\rceil$.

\begin{definition}[$\alpha$-admissible unit vectors]\label{def:alpha-admissible}
Fix $\alpha\in[1,d]$ and write $k:=\lfloor\alpha\rfloor$, $\theta:=\alpha-k\in[0,1)$, and $r:=\lceil\alpha\rceil$.
A unit vector $\psi\in\mathds C^n\otimes\mathds C^m$ with ordered Schmidt coefficients
$s_1(\psi)\ge\cdots\ge s_d(\psi)$ is \emph{$\alpha$-admissible} if:
\begin{enumerate}
\item[(i)] (\emph{Rank ceiling}) $s_j(\psi)=0$ for all $j\ge r+1$ (equivalently $\SR(\psi)\le r$);
\item[(ii)] (\emph{Normalized ratio}) if $\theta>0$, then
\begin{align}\label{normaliz}
s_{k+1}(\psi)\ \le\ \frac{\theta}{k}\sum_{j=1}^k s_j(\psi).
\end{align}
\end{enumerate}
\end{definition}
\noindent
Let \(\mathcal V_\alpha\) denote the set of all \(\alpha\)-admissible unit vectors in \(\mathds C^n\otimes\mathds C^m\).

\begin{remark}
\label{rem:alpha-meaning}
\begin{enumerate}
\item[\textup{(i)}] If $\alpha=k$ is an integer, then $\theta=0$ and $r=k$, so Definition~\ref{def:alpha-admissible}
reduces to $\SR(\psi)\le k$. Hence
\[
\mathcal{V}_k=\{\psi:\ \|\psi\|=1,\ \SR(\psi)\le k\}.
\]

\item[\textup{(ii)}] If $\alpha\in(k,k+1)$, then $r=k+1$, so $\psi\in\mathcal{V}_\alpha$ if and only if
$\SR(\psi)\le k+1$ and
\[
\frac{k\,s_{k+1}(\psi)}{\sum_{j=1}^k s_j(\psi)}\ \le\ \theta.
\]
Equivalently, $\alpha=k+\theta$ allows at most one additional Schmidt coefficient beyond level $k$, and
$\theta$ controls its size relative to the leading $k$ coefficients.

\item[\textup{(iii)}] For every unit vector with $\SR(\psi)\le k+1$ one always has
\[
s_{k+1}(\psi)\ \le\ \frac{1}{k}\sum_{j=1}^k s_j(\psi)
\qquad\big(\text{since } \sum_{j=1}^k s_j(\psi)\ge k\,s_{k+1}(\psi)\big).
\]
Hence as $\theta\uparrow 1$, the ratio constraint becomes automatic and $\mathcal{V}_{k+\theta}$ increases
all the way up to the full rank-$\le k+1$ class. In particular, there is no plateau (``dead zone'') for $\alpha\in(k,k+1)$.

\item[\textup{(iv)}] If   $X=\mathrm{mat}(\psi)\in \mathbb{M}_{n,m}$ is  a fixed matricization, then \eqref{normaliz} can be written as
\begin{align}
\label{eq:norm-ratio}
\frac{\|X\|_1}{\|X\|_{(k)}}\ \le\ 1+\frac{\theta}{k}\,\,\, \text{or equivalently }\,\,\, \sigma_{k+1}(X)\ \le\ \frac{\theta}{k}\sum_{j=1}^k \sigma_j(X),
\end{align}
where $\|\cdot\|_{(k)}$ are  the Ky-Fan norms.
\end{enumerate}

\end{remark}

Now we define fractional cones.

\begin{definition}
  We define the fractional  $\alpha$-superpositive cone as
  \begin{align}\label{cone}
    \mathsf{K}_\alpha:=\overline{\mathrm{cone}}\{\,\psi\psi^*:\ \psi\in\mathcal{V}_\alpha\,\}\subseteq (\mathbb{M}_n\otimes \mathbb{M}_m)^+,
  \end{align}
  and its dual, the fractional $\alpha$-block-positive  cone as
  \begin{align}\label{alpha-block-cone}
\mathsf{BP}_\alpha:=\mathsf{K}_\alpha^\ast
=\{\,W\in (\mathbb{M}_n\otimes \mathbb{M}_m)^{\mathrm{h}}:\ \langle\psi,W\psi\rangle\ge 0\ \text{for all }\psi\in\mathcal{V}_\alpha\,\}.
\end{align}
\end{definition}
\noindent
In words, $\mathsf K_\alpha$ consists of all positive semidefinite operators that can be approximated by conic combinations of rank-one projectors onto $\alpha$-admissible vectors; it interpolates between the separable cone and the full positive cone as $\alpha$ increases.


We will give basic properties of our fractional cones here.
For $\alpha\in[1,d]$ and $W\in (\mathbb{M}_n\otimes \mathbb{M}_m)^{\mathrm{h}}$, define
\begin{equation}
\label{eq:lambda-alpha}
\lambda_\alpha(W)\;:=\;\min\Big\{\langle x,\,W x\rangle:\ x\in\mathcal{V}_\alpha\Big\}.
\end{equation}

\begin{proposition}
\label{prop:compact-lambda}
Fix $\alpha\in[1,d]$ and write $k:=\lfloor\alpha\rfloor$ and $\theta:=\alpha-k\in[0,1)$.
Then:
\begin{enumerate}
\item[\textup{(a)}] The set $\mathcal{V}_\alpha$ is a nonempty compact subset of the unit sphere of
$\mathbb{C}^n\otimes\mathbb{C}^m$.
\item[\textup{(b)}] For every $W\in (\mathbb{M}_n\otimes \mathbb{M}_m)^{\mathrm{h}}$, the minimum in \eqref{eq:lambda-alpha} is attained.
Moreover, the map $W\mapsto \lambda_\alpha(W)$ is concave and Lipschitz (hence continuous).
\item[\textup{(c)}] $\mathsf{BP}_\alpha$ is a closed convex cone in $(\mathbb{M}_n\otimes \mathbb{M}_m)^{\mathrm{h}}$. In particular,
\[
W\in \mathsf{BP}_\alpha
\quad\Longleftrightarrow\quad
\lambda_\alpha(W)\ge 0.
\]
\end{enumerate}
\end{proposition}

\begin{proof}
\noindent\textup{(a)} Because any unit simple tensor $u\otimes v$ has Schmidt rank $1$, hence lies in $\mathcal{V}_\alpha$
for every $\alpha\in[1,d]$ and so each $\mathcal{V}_\alpha$ is nonempty.

Let $\mathbb{S}$ denote the unit sphere of $\mathbb{C}^n\otimes\mathbb{C}^m$, which is compact.
It suffices to show that $\mathcal{V}_\alpha$ is closed in $\mathbb{S}$.

Consider the continuous map $x\mapsto X=\mathrm{mat}(x)\in \mathbb{M}_{n,m}$ given by a fixed matricization.
The Schmidt coefficients of $x$ coincide with the singular values of $X$; in particular, each function
$x\mapsto s_j(x)$ is continuous on $\mathbb{S}$  \cite[Ch.~II]{BhatiaMatrix}(singular values depend continuously on the matrix entries).

If $\theta=0$, then $\mathcal{V}_k=\{x\in\mathbb{S}:\mathrm{rank}(\mathrm{mat}(x))\le k\}$.
The rank constraint is closed because it is equivalent to the vanishing of all $(k+1)\times(k+1)$ minors, which are polynomial
functions of the entries.

If $\theta\in(0,1)$, then $\mathcal{V}_\alpha$ is the intersection of the closed set
$\{x\in\mathbb{S}: s_{k+2}(x)=0\}$ (equivalently $\mathrm{rank}(\mathrm{mat}(x))\le k+1$) with the closed set
$\{x\in\mathbb{S}: s_{k+1}(x)-\frac{\theta}{k}\sum_{j=1}^k s_j(x)\le 0\}$, since the left-hand side is continuous in $x$.
Hence $\mathcal{V}_\alpha$ is closed in $\mathbb{S}$.

Therefore $\mathcal{V}_\alpha$ is a closed subset of a compact set, and is compact.

\medskip
\noindent\textup{(b)} For fixed $W$, the function $x\mapsto \langle x,Wx\rangle$ is continuous on the compact set $\mathcal{V}_\alpha$,
so the minimum in \eqref{eq:lambda-alpha} is attained.

Concavity in $W$ follows because $\lambda_\alpha$ is a pointwise minimum of linear functionals:
for $t\in[0,1]$,
\[
\lambda_\alpha(tW_1+(1-t)W_2)
=
\min_{x\in\mathcal{V}_\alpha}\big(t\langle x,W_1x\rangle+(1-t)\langle x,W_2x\rangle\big)
\ge
t\lambda_\alpha(W_1)+(1-t)\lambda_\alpha(W_2).
\]
Finally, $\lambda_\alpha$ is Lipschitz with respect to the operator norm:
if $\|W-W'\|\le\varepsilon$, then for all unit $x$ we have
$|\langle x,Wx\rangle-\langle x,W'x\rangle|\le\varepsilon$, hence
$|\lambda_\alpha(W)-\lambda_\alpha(W')|\le\varepsilon$.

\medskip
\noindent\textup{(c)} By definition,
\[
W\in\mathsf{BP}_\alpha
\quad\Longleftrightarrow\quad
\langle x,Wx\rangle\ge 0\ \ \text{for all }x\in\mathcal{V}_\alpha.
\]
This is equivalent to $\min_{x\in\mathcal{V}_\alpha}\langle x,Wx\rangle\ge 0$, i.e.\ $\lambda_\alpha(W)\ge 0$.

Since $\mathsf{BP}_\alpha=\{W:\lambda_\alpha(W)\ge 0\}$ and $\lambda_\alpha$ is concave and continuous, it follows that
$\mathsf{BP}_\alpha$ is a closed convex cone in $(\mathbb{M}_n\otimes \mathbb{M}_m)^{\mathrm{h}}$.
\end{proof}


As a consequence of the compactness of $\mathcal{V}_\alpha$, we  conclude that the convex cone generated by $\mathcal{V}_\alpha$ is automatically closed.

\begin{lemma}
\label{lem:cone-closed-trace1}
Let $S\subset (\mathbb{M}_n\otimes \mathbb{M}_m)^{\mathrm{h}}$ be compact and assume that
$\Tr(H)=1$ for all $H\in S$.
Define the (conic) hull
\[
\cone(S):=\Big\{\sum_{i=1}^r t_i H_i:\ r\in\mathbb{N},\ t_i\ge 0,\ H_i\in S\Big\}.
\]
Then $\cone(S)$ is a closed convex cone. Moreover,
\begin{equation}\label{eq:cone-as-scaling}
\cone(S)=\{\,t\rho:\ t\ge 0,\ \rho\in \conv(S)\,\}.
\end{equation}
In particular, $\overline{\cone(S)}=\cone(S)$.
\end{lemma}

\begin{proof}
We work in the finite-dimensional real vector space
$V:=(\mathbb{M}_n\otimes \mathbb{M}_m)^{\mathrm{h}}$ equipped with any norm.
Since $S$ is compact, its convex hull $\conv(S)$ is compact as well
(see, e.g., \cite[Thm.~1.1.11]{SchneiderCB}; this is a standard consequence of
Carath\'eodory's theorem \cite[Thm.~1.1.4]{SchneiderCB}).

\noindent First note that  \eqref{eq:cone-as-scaling} follows clearly form the convex hull definition.

\noindent
Let $(Y_\ell)_{\ell\in\mathbb{N}}\subset \cone(S)$ be a sequence with $Y_\ell\to Y$ in $V$.
Using \eqref{eq:cone-as-scaling}, write
\[
Y_\ell=t_\ell \rho_\ell,\qquad t_\ell\ge 0,\ \rho_\ell\in \conv(S).
\]
Taking traces and using $\Tr(\rho_\ell)=1$ (since $\conv(S)$ is contained in the affine hyperplane $\{\Tr=1\}$),
we obtain
\[
t_\ell=\Tr(Y_\ell)\ \longrightarrow\ \Tr(Y)=:t.
\]
In particular $t\ge 0$.

If $t=0$, then $t_\ell\to 0$. Since $\conv(S)$ is compact, it is bounded:
there exists $M>0$ with $\|\rho\|\le M$ for all $\rho\in\conv(S)$.
Hence $\|Y_\ell\|=\|t_\ell\rho_\ell\|\le t_\ell M\to 0$, so $Y=0\in\cone(S)$.

If $t>0$, compactness of $\conv(S)$ implies sequential compactness, so after passing to a subsequence
we may assume $\rho_{\ell_j}\to \rho\in\conv(S)$. Then
\[
Y_{\ell_j}=t_{\ell_j}\rho_{\ell_j}\ \longrightarrow\ t\rho,
\]
and by uniqueness of limits we have $Y=t\rho\in\cone(S)$.

Thus every convergent sequence in $\cone(S)$ converges to a point of $\cone(S)$, i.e.\ $\cone(S)$ is closed.
Convexity is immediate from the definition of $\cone(S)$, so $\cone(S)$ is a closed convex cone.
\end{proof}

\begin{corollary}
\label{cor:Kalpha-no-closure}
With $\mathcal{S}_\alpha:=\{x x^\ast:\ x\in\mathcal{V}_\alpha\}$ one has
\[
\mathsf{K}_\alpha=\mathrm{cone}(\mathcal{S}_\alpha).
\]
\end{corollary}

\begin{proof}
By Proposition~\ref{prop:compact-lambda}\textup{(a)}, the set $\mathcal V_\alpha$ is compact.
The map $x\mapsto xx^*$ is continuous and $\Tr(xx^*)=\|x\|^2=1$, hence $\mathcal S_\alpha$ is compact and trace-$1$.
Lemma~\ref{lem:cone-closed-trace1} gives $\overline{\cone(\mathcal S_\alpha)}=\cone(\mathcal S_\alpha)$.
Since $\mathsf K_\alpha=\overline{\cone(\mathcal S_\alpha)}$ by \eqref{cone}, the claim follows.
\end{proof}

The next result shows that the interpolation is nontrivial: between $k$ and $k+1$ one obtains genuinely new cones.

\begin{proposition}
\label{prop:strict-inclusions}
Fix an integer $k\in\{1,\dots,d-1\}$ and $\theta\in(0,1)$, and set $\alpha:=k+\theta$.
Then:
\begin{enumerate}
\item[\textup{(a)}] $\mathcal{V}_k \subsetneq \mathcal{V}_\alpha \subsetneq \mathcal{V}_{k+1}$.
\item[\textup{(b)}] $\mathsf{K}_k \subsetneq \mathsf{K}_\alpha \subsetneq \mathsf{K}_{k+1}$.
\item[\textup{(c)}] $\mathsf{BP}_k \supsetneq \mathsf{BP}_\alpha \supsetneq \mathsf{BP}_{k+1}$.
\end{enumerate}
\end{proposition}

\begin{proof}
\textup{(a)} By definition, $\mathcal{V}_k$ consists of unit vectors with Schmidt rank at most $k$.
Thus any $\psi\in\mathcal{V}_k$ satisfies $s_{k+1}(\psi)=0$ and $s_j(\psi)=0$ for all $j\ge k+1$.
In particular, $\psi$ satisfies the rank ceiling $\SR(\psi)\le k+1$ and also
\[
s_{k+1}(\psi)=0 \le \frac{\theta}{k}\sum_{j=1}^k s_j(\psi),
\]
so $\mathcal{V}_k\subseteq \mathcal{V}_\alpha$.
Also, $\mathcal{V}_\alpha\subseteq \mathcal{V}_{k+1}$ is immediate from the rank ceiling in
Definition~\ref{def:alpha-admissible} (for $\alpha\in(k,k+1)$ this ceiling is $\SR(\psi)\le k+1$).

We now prove both inclusions are strict.
Choose orthonormal families $\{e_i\}_{i=1}^{k+1}\subset\mathbb{C}^n$ and $\{f_i\}_{i=1}^{k+1}\subset\mathbb{C}^m$.
Set
\[
a:=\frac{1}{\sqrt{k+\theta^2}},
\qquad
b:=\frac{\theta}{\sqrt{k+\theta^2}}.
\]
Define
\begin{equation}
\label{eq:psi-theta}
\psi_\theta:=\sum_{i=1}^{k} a\, e_i\otimes f_i \;+\; b\, e_{k+1}\otimes f_{k+1}.
\end{equation}
Then $\|\psi_\theta\|^2=ka^2+b^2=1$.
The Schmidt coefficients of $\psi_\theta$ are $a$ (with multiplicity $k$) and $b$ (with multiplicity $1$),
so $\SR(\psi_\theta)=k+1$ and
\[
s_{k+1}(\psi_\theta)=b=\theta a=\frac{\theta}{k}\sum_{j=1}^k s_j(\psi_\theta).
\]
Hence $\psi_\theta\in\mathcal{V}_\alpha$ but $\psi_\theta\notin\mathcal{V}_k$ (since $\SR(\psi_\theta)=k+1$),
so $\mathcal{V}_k\subsetneq\mathcal{V}_\alpha$.

For the second strictness, set $\theta':=(1+\theta)/2\in(\theta,1)$ and define
\[
a':=\frac{1}{\sqrt{k+(\theta')^2}},
\qquad
b':=\frac{\theta'}{\sqrt{k+(\theta')^2}},
\qquad
\psi_{\theta'}:=\sum_{i=1}^{k} a'\, e_i\otimes f_i \;+\; b'\, e_{k+1}\otimes f_{k+1}.
\]
Again $\|\psi_{\theta'}\|=1$ and $\SR(\psi_{\theta'})=k+1$, so $\psi_{\theta'}\in\mathcal{V}_{k+1}$.
However,
\[
s_{k+1}(\psi_{\theta'})=b'=\theta' a' > \theta a'=\frac{\theta}{k}\sum_{j=1}^k s_j(\psi_{\theta'}),
\]
so $\psi_{\theta'}\notin\mathcal{V}_\alpha$. Therefore $\mathcal{V}_\alpha\subsetneq\mathcal{V}_{k+1}$.

\medskip
\textup{(b)}
For $\beta\in\{k,\alpha,k+1\}$ set
\[
S_\beta:=\{\,xx^\ast:\ x\in\mathcal{V}_\beta\,\}\subset (\mathbb{M}_n\otimes \mathbb{M}_m)^+ .
\]
Since $\mathcal{V}_\beta$ is a closed subset of the unit sphere (hence compact) and
$x\mapsto xx^\ast$ is continuous, each $S_\beta$ is compact. Moreover, $\Tr(xx^\ast)=\|x\|^2=1$ for all
$x\in\mathcal{V}_\beta$, hence $\Tr(H)=1$ for all $H\in S_\beta$.

We record two elementary claims.

 claim 1.
Let $S\subset (\mathbb{M}_n\otimes \mathbb{M}_m)^{\mathrm{h}}$ be compact and assume $\Tr(H)=1$ for all $H\in S$.
If $\rho\in \overline{\cone}(S)$ and $\Tr(\rho)=1$, then $\rho\in \conv(S)$.

Proof of Claim~1.
Pick $\rho_j\in \cone(S)$ with $\rho_j\to \rho$. Write $\rho_j=t_j \sigma_j$ where
\[
t_j:=\Tr(\rho_j)\ge 0,
\qquad
\sigma_j:=\rho_j/t_j \quad (t_j>0).
\]
Since $\rho_j\in\cone(S)$ and every element of $S$ has trace $1$, we have $\Tr(\sigma_j)=1$ and hence
$\sigma_j\in \conv(S)$ (indeed, $\rho_j=\sum_i a_{i,j}H_{i,j}$ with $a_{i,j}\ge 0$, $H_{i,j}\in S$, so
$t_j=\sum_i a_{i,j}$ and $\sigma_j=\sum_i (a_{i,j}/t_j)H_{i,j}$ is a convex combination).
By continuity of the trace, $t_j=\Tr(\rho_j)\to \Tr(\rho)=1$, so for all large $j$ we have $t_j>0$ and
\[
\sigma_j=\rho_j/t_j \longrightarrow \rho .
\]
Finally, $\conv(S)$ is compact (hence closed) because $S$ is compact in finite dimension
(see, e.g., \cite[Thm.~1.1.11]{SchneiderCB}), so $\rho\in\conv(S)$.

 claim~2.
Let $\mathcal{T}$ be a set of rank-one positive semidefinite matrices of trace $1$.
If $\rho$ is rank-one, $\rho\geq  0$, $\Tr(\rho)=1$, and $\rho\in \conv(\mathcal{T})$, then in fact $\rho\in \mathcal{T}$.

 proof of claim~2.
Write $\rho=\sum_{i=1}^N p_i\,\tau_i$ with $p_i>0$, $\sum_i p_i=1$, and $\tau_i\in \mathcal{T}$.
Let $v$ be any unit vector orthogonal to $\mathrm{Ran}(\rho)$; since $\rho$ has rank one this means
$v\perp \psi$ where $\rho=\psi\psi^\ast$ with $\|\psi\|=1$. Then
\[
0=\langle v,\rho v\rangle=\sum_{i=1}^N p_i\,\langle v,\tau_i v\rangle.
\]
Each $\tau_i\geq 0$, hence each term $\langle v,\tau_i v\rangle\ge 0$, and since $p_i>0$ we get
$\langle v,\tau_i v\rangle=0$ for all $i$. For $\tau_i=\varphi_i\varphi_i^\ast$ this means
$0=\langle v,\varphi_i\varphi_i^\ast v\rangle=|\langle v,\varphi_i\rangle|^2$, so $v\perp \varphi_i$.

Since $v\perp\psi \Rightarrow v\perp\varphi_i$, we have
$\mathrm{span}\{\psi\}^\perp \subseteq \mathrm{span}\{\varphi_i\}^\perp$.
Taking orthogonal complements yields
$\mathrm{span}\{\varphi_i\}\subseteq \mathrm{span}\{\psi\}$, so $\varphi_i=c_i\psi$.
As $\|\varphi_i\|=\|\psi\|=1$, we have $|c_i|=1$ and hence
$\tau_i=\varphi_i\varphi_i^\ast=\psi\psi^\ast=\rho$. Therefore $\rho\in \mathcal{T}$.

\medskip
\noindent
We now prove the strict inclusions in part (b). From part \textup{(a)} we have
$\mathcal{V}_k\subseteq \mathcal{V}_\alpha\subseteq \mathcal{V}_{k+1}$, hence
$S_k\subseteq S_\alpha\subseteq S_{k+1}$ and therefore
\[
\mathsf{K}_k=\overline{\cone}(S_k)\subseteq \overline{\cone}(S_\alpha)=\mathsf{K}_\alpha
\subseteq \overline{\cone}(S_{k+1})=\mathsf{K}_{k+1}.
\]

\medskip
\noindent
To see strictness of $\mathsf{K}_k\subsetneq \mathsf{K}_\alpha$, let $\psi_\theta$ be the unit vector constructed in
part \textup{(a)} so that $\psi_\theta\in \mathcal{V}_\alpha\setminus \mathcal{V}_k$, and set
$\rho_\theta:=\psi_\theta\psi_\theta^*$.
Then $\rho_\theta\in S_\alpha\subseteq \mathsf{K}_\alpha$ and $\Tr(\rho_\theta)=1$.
If $\rho_\theta\in\mathsf{K}_k=\overline{\cone}(S_k)$, Claim~1 gives
$\rho_\theta\in \conv(S_k)$. Since $\rho_\theta$ is rank one, Claim~2 yields
$\rho_\theta\in S_k$, i.e.\ $\psi_\theta\in \mathcal{V}_k$, a contradiction. Hence
$\rho_\theta\notin \mathsf{K}_k$ and $\mathsf{K}_k\subsetneq \mathsf{K}_\alpha$.

\medskip
\noindent
Finally, strictness of $\mathsf{K}_\alpha\subsetneq \mathsf{K}_{k+1}$ is analogous.
Let $\psi_{\theta'}$ be the unit vector from part \textup{(a)} with
$\psi_{\theta'}\in \mathcal{V}_{k+1}\setminus \mathcal{V}_\alpha$, and set
$\rho_{\theta'}:=\psi_{\theta'}\psi_{\theta'}^*$.
Then $\rho_{\theta'}\in S_{k+1}\subseteq \mathsf{K}_{k+1}$ and $\Tr(\rho_{\theta'})=1$.
If $\rho_{\theta'}\in\mathsf{K}_\alpha=\overline{\cone}(S_\alpha)$, then Claim~1
gives $\rho_{\theta'}\in \conv(S_\alpha)$, and by Claim~2 we get
$\rho_{\theta'}\in S_\alpha$, i.e.\ $\psi_{\theta'}\in\mathcal{V}_\alpha$, a contradiction.
Thus $\rho_{\theta'}\notin \mathsf{K}_\alpha$ and $\mathsf{K}_\alpha\subsetneq \mathsf{K}_{k+1}$.

\medskip
\textup{(c)} We use the following finite-dimensional duality fact:
if $\mathcal{K}\subsetneq \mathcal{L}$ are closed convex cones in a finite-dimensional real inner product space, then $\mathcal{L}^\ast\subsetneq \mathcal{K}^\ast$. We emphasize that $\mathcal{K}^*$ denotes the \emph{dual cone} in the dual space $V^*$:
\[
\mathcal{K}^*=\{f\in V^*:\ f(y)\ge 0\ \forall y\in \mathcal{K}\},
\]
not the full dual space $V^*$ of all linear functionals.
Indeed, pick $x\in \mathcal{L}\setminus \mathcal{K}$. Since $\mathcal{K}$ is closed, convex, and contains $0$, the strong separation theorem yields a nonzero
linear functional $f$ such that $f(y)\ge 0$ for all $y\in \mathcal{K}$ but $f(x)<0$.
Thus $f\in \mathcal{K}^\ast$ and $f\notin \mathcal{L}^\ast$, proving $\mathcal{L}^\ast\subsetneq \mathcal{K}^\ast$.
Applying this to $\mathcal{K}=\mathsf{K}_k$, $\mathcal{L}=\mathsf{K}_\alpha$ and then to $\mathcal{K}=\mathsf{K}_\alpha$, $\mathcal{L}=\mathsf{K}_{k+1}$, and using
$\mathsf{BP}_\beta=\mathsf{K}_\beta^\ast$, we obtain
\[
\mathsf{BP}_k\supsetneq \mathsf{BP}_\alpha\supsetneq \mathsf{BP}_{k+1}.
\]
\end{proof}

\begin{proposition}\label{prop:local-unitary-invariance}
Fix $\alpha\in[1,d]$. Let $U\in \mathbb{M}_n$ and $V\in \mathbb{M}_m$ be unitary. Then:
\begin{align*}
&\mathrm{(i)}\,\,
\psi\in\mathcal{V}_\alpha
\quad\Longleftrightarrow\quad
(U\otimes V)\psi\in\mathcal{V}_\alpha\qquad\qquad\qquad \text{for any unit vector $\psi\in\mathds{C}^n\otimes\mathds{C}^m$}\\
&\mathrm{(ii)}\,\,
\rho\in\mathsf{K}_\alpha
\quad\Longleftrightarrow\quad
(U\otimes V)\,\rho\,(U\otimes V)^\ast\in\mathsf{K}_\alpha\qquad\qquad\text{for any $\rho\in \mathbb{M}_n\otimes \mathbb{M}_m$,}\\
&\mathrm{(iii)}\,\,
W\in\mathsf{BP}_\alpha
\quad\Longleftrightarrow\quad
(U\otimes V)\,W\,(U\otimes V)^\ast\in\mathsf{BP}_\alpha\qquad\text{for any $W\in (\mathbb{M}_n\otimes \mathbb{M}_m)^{\mathrm{h}}$}
\end{align*}
\end{proposition}

\subsection{\texorpdfstring{$\alpha$-positive and $\alpha$-superpositive maps}{Alpha-positive and alpha-superpositive maps}}
\label{subsec:alpha-maps}

Let $\Phi:\mathbb{M}_n\to \mathbb{M}_m$ be Hermitian-preserving and let $C_\Phi\in (\mathbb{M}_n\otimes \mathbb{M}_m)^{\mathrm{h}}$ denote its Choi matrix.

\begin{definition}[$\alpha$-positive maps]
\label{def:P-alpha}
For $\alpha\in[1,d]$, we say that a Hermitian-preserving map  $\Phi:\mathbb{M}_n\to \mathbb{M}_m$   is \emph{$\alpha$-positive} if $C_\Phi \in \mathsf{BP}_\alpha$.
Equivalently,
\[
\langle \psi,\, C_\Phi \psi\rangle \ge 0
\qquad\text{for all }\psi\in\mathcal{V}_\alpha,
\]
where $\mathcal{V}_\alpha$ is the set of $\alpha$-admissible vectors.
We denote the cone of $\alpha$-positive maps by $\mathsf{P}_\alpha$.

In addition,  we say that a Hermitian-preserving map $\Phi:\mathbb{M}_n\to \mathbb{M}_m$  is \emph{$\alpha$-superpositive} if $C_\Phi \in \mathsf{K}_\alpha$.
We denote the cone of $\alpha$-superpositive maps by $\mathsf{SP}_\alpha$.
\end{definition}

\begin{remark}\label{rem:maps}
Via the Choi identification, since the cones $\mathsf{K}_\alpha$ and $\mathsf{BP}_\alpha$ are dual,   consequently
\[
\mathsf{SP}_\alpha = \mathsf{P}_\alpha^\ast,
\qquad
\mathsf{P}_\alpha = \mathsf{SP}_\alpha^\ast,
\]
where duals are taken with respect to $\langle\cdot,\cdot\rangle_{\mathrm{Choi}}$. Moreover, it follows from Proposition~\ref{prop:strict-inclusions} that  $\mathsf{P}_k \supsetneq \mathsf{P}_\alpha \supsetneq \mathsf{P}_{k+1}$ and
$\mathsf{SP}_k \subsetneq \mathsf{SP}_\alpha \subsetneq \mathsf{SP}_{k+1}$.
\end{remark}

\begin{proposition}
\label{cor:unitary-covariance-basis}
Fix $\alpha\in[1,d]$ and let $\Phi:\mathbb{M}_n\to \mathbb{M}_m$ be Hermitian-preserving.
\begin{enumerate}
\item[\textup{(a)}] For all unitaries $U\in \mathbb{M}_m$, $V\in \mathbb{M}_n$,
\[
\Phi\in\mathsf{P}_\alpha
\quad\Longleftrightarrow\quad
\mathrm{Ad}_U\circ \Phi\circ \mathrm{Ad}_V\in\mathsf{P}_\alpha
\quad\text{and}\quad
\Phi\in\mathsf{SP}_\alpha
\quad\Longleftrightarrow\quad
\mathrm{Ad}_U\circ \Phi\circ \mathrm{Ad}_V\in\mathsf{SP}_\alpha.
\]
\item[\textup{(b)}] The membership of $\Phi$ in $\mathcal{P}_\alpha$ (and in $\mathcal{SP}_\alpha$) does not depend on the choice
of matrix units used to define $C_\Phi$.
\end{enumerate}
\end{proposition}

\begin{proof}
\textup{(a)} The Choi matrix $C_{\mathrm{Ad}_U\circ \Phi\circ \mathrm{Ad}_V}$ is obtained from $C_\Phi$ by
conjugation with the local unitary $V^{\mathsf T}\otimes U$.
Since $V^{\mathsf T}$ is unitary, Proposition~\ref{prop:local-unitary-invariance}\textup{(c)} implies
$C_\Phi\in\mathsf{BP}_\alpha$ if and only if $C_{\mathrm{Ad}_U\circ \Phi\circ \mathrm{Ad}_V}\in\mathsf{BP}_\alpha$,
i.e.\ $\Phi\in\mathsf{P}_\alpha$ if and only if $\mathrm{Ad}_U\circ \Phi\circ \mathrm{Ad}_V\in\mathsf{P}_\alpha$.
The statement for $\mathsf{SP}_\alpha$ follows similarly using invariance of $\mathsf{K}_\alpha$.

\textup{(b)} Changing matrix units by a unitary $V$ (i.e.\ changing the basis of $\mathds{C}^n$) transforms the Choi matrix to
$(V^{\mathsf T}\otimes I)\,C_\Phi\,(V^{\mathsf T}\otimes I)^\ast$.
By Proposition~\ref{prop:local-unitary-invariance}\textup{(c)}, membership in $\mathsf{BP}_\alpha$ (and in $\mathsf{K}_\alpha$) is preserved,
hence so is membership in $\mathsf{P}_\alpha$ (and in $\mathsf{SP}_\alpha$).
\end{proof}

\medskip
\noindent
At integer levels, the familiar cones associated with $k$-positivity enjoy strong stability properties (e.g.,
compatibility with natural compositions and cone operations).  A key novelty of the fractional interpolation is that
some of these structural features break down away from integers.  The next proposition exhibits a concrete failure of
CP post-composition stability for non-integer $\alpha$, emphasizing that $\mathsf P_\alpha$ is genuinely new and not
a mere reformulation of the integer theory.

\begin{proposition}
\label{prop:failure-cp-new}
Fix an integer $k\in\{1,\dots,d-1\}$ and $\theta\in(0,1)$, and set $\alpha:=k+\theta$.
Then $\mathsf{P}_\alpha$ is not stable under CP post-composition: there exist a Hermitian-preserving
$\alpha$-positive map $\Phi:\mathbb{M}_n\to \mathbb{M}_m$ and a completely positive map $\Gamma:\mathbb{M}_m\to \mathbb{M}_m$ such that
$\Gamma\circ \Phi\notin\mathsf{P}_\alpha$.

More precisely, if $\Phi\in\mathsf{P}_\alpha\setminus \mathsf{P}_{k+1}$, then there exists a one-Kraus CP map
$\Gamma=\mathrm{Ad}_A$ with $\mathrm{Ad}_A\circ\Phi\notin\mathsf{P}_\alpha$.
\end{proposition}

\begin{proof}
Since $\mathsf{P}_{k+1}\subsetneq \mathsf{P}_\alpha$ by Remark~\ref{rem:maps},
we may fix $\Phi\in\mathsf{P}_\alpha\setminus \mathsf{P}_{k+1}$.
Let $W:=C_\Phi\in (\mathbb{M}_n\otimes \mathbb{M}_m)^{\mathrm{h}}$ be its Choi matrix.
Then $W\in \mathsf{BP}_\alpha$ but $W\notin \mathsf{BP}_{k+1}$.

\smallskip
\emph{Step 1: a rank-$(k+1)$ witness vector.}
Since $W\notin \mathsf{BP}_{k+1}$, there exists a unit vector $\varphi\in \mathds{C}^n\otimes\mathds{C}^m$
with $\SR(\varphi)\le k+1$ such that
\begin{equation}
\label{eq:neg-on-phi}
\langle \varphi, W\varphi\rangle<0.
\end{equation}
Necessarily $\SR(\varphi)=k+1$ (otherwise $\SR(\varphi)\le k$ would imply $\varphi\in\mathcal{V}_k\subseteq \mathcal{V}_\alpha$,
contradicting $W\in\mathsf{BP}_\alpha$).
Fix a Schmidt decomposition
\[
\varphi=\sum_{i=1}^{k+1} s_i\, e_i\otimes f_i,
\qquad s_1\ge \cdots\ge s_{k+1}>0,
\qquad \sum_{i=1}^{k+1}s_i^2=1,
\]
with orthonormal families $\{e_i\}_{i=1}^{k+1}\subset\mathds{C}^n$ and $\{f_i\}_{i=1}^{k+1}\subset\mathds{C}^m$.
Set $S:=\sum_{i=1}^k s_i>0$.

\smallskip
\emph{Step 2: attenuate the last Schmidt coefficient to enter $\mathcal{V}_\alpha$.}
Choose $t>0$ such that
\begin{equation}
\label{eq:t-choice}
\frac{s_{k+1}}{t}\ \le\ \frac{\theta}{k}\,S.
\end{equation}
Define $A_t\in \mathbb{M}_m$ by
\[
A_t f_i=f_i\ (1\le i\le k),\qquad A_t f_{k+1}=t\,f_{k+1},
\qquad\text{and}\qquad A_t|_{\mathrm{span}\{f_1,\dots,f_{k+1}\}^\perp}=I,
\]
and let $\Gamma_t:=\mathrm{Ad}_{A_t}$, which is completely positive.

Now define the (unnormalized) vector
\[
\widetilde{\psi}_t:=\sum_{i=1}^{k} s_i\,e_i\otimes f_i\;+\;\frac{s_{k+1}}{t}\,e_{k+1}\otimes f_{k+1},
\qquad
\psi_t:=\frac{\widetilde{\psi}_t}{\|\widetilde{\psi}_t\|}.
\]
Then $\SR(\psi_t)=k+1$, and $\psi_t$ is already in Schmidt form with Schmidt coefficients
\[
s_i(\psi_t)=\frac{s_i}{\|\widetilde{\psi}_t\|}\ (1\le i\le k),
\qquad
s_{k+1}(\psi_t)=\frac{s_{k+1}/t}{\|\widetilde{\psi}_t\|}.
\]
Hence, using \eqref{eq:t-choice},
\[
s_{k+1}(\psi_t)
=\frac{s_{k+1}/t}{\|\widetilde{\psi}_t\|}
\le
\frac{\theta}{k}\,\frac{\sum_{i=1}^k s_i}{\|\widetilde{\psi}_t\|}
=
\frac{\theta}{k}\sum_{i=1}^k s_i(\psi_t).
\]
Therefore $\psi_t\in\mathcal{V}_\alpha$ by Definition~\ref{def:alpha-admissible}.
Moreover,
\begin{align}\label{q1}
(I\otimes A_t^\ast)\widetilde{\psi}_t=\varphi,
\qquad\text{hence}\qquad
(I\otimes A_t^\ast)\psi_t=\frac{1}{\|\widetilde{\psi}_t\|}\,\varphi.
\end{align}

\smallskip
\emph{Step 3: post-compose and test.}
With a direct computation on the Choi matrix, we see that
\begin{align}\label{choi-post-compos}
C_{\Gamma_t\circ\Phi}=(I\otimes \Gamma_t)(W)=(I\otimes A_t)\,W\,(I\otimes A_t^\ast).
\end{align}
Therefore, using \eqref{q1} and \eqref{choi-post-compos} we get
\[
\langle \psi_t,\, C_{\Gamma_t\circ\Phi}\,\psi_t\rangle
=
\langle (I\otimes A_t^\ast)\psi_t,\ W\,(I\otimes A_t^\ast)\psi_t\rangle
=
\frac{1}{\|\widetilde{\psi}_t\|^2}\,\langle \varphi,W\varphi\rangle.
\]
By \eqref{eq:neg-on-phi}, the right-hand side is strictly negative. Since $\psi_t\in\mathcal{V}_\alpha$,
this shows $C_{\Gamma_t\circ\Phi}\notin \mathsf{BP}_\alpha$, i.e.\ $\Gamma_t\circ\Phi\notin\mathsf{P}_\alpha$.

Thus CP post-composition can destroy $\alpha$-positivity for $\alpha\in(k,k+1)$.
\end{proof}

\begin{remark}
For integer $\alpha=k$, the cone $\mathsf{P}_k$ coincides with $k$-positivity and is stable under CP pre- and post-composition.
The above phenomenon is therefore genuinely fractional: it occurs only for non-integer $\alpha$.
\end{remark}


\subsection{\texorpdfstring{A sharp threshold example: $\alpha$-positivity of $\Phi_t(X)=\Tr(X)I-tX$}{A sharp threshold example: alpha-positivity of Phi_t}}
\label{subsec:flagship-depolarizing}

In this subsection we compute the $\alpha$-positivity region for the unitarily covariant one-parameter family
\begin{equation}
\label{eq:Phi-t}
\Phi_t:\mathbb{M}_d\to \mathbb{M}_d,\qquad \Phi_t(X):=\Tr(X)\,I_d-tX,
\end{equation}
which is Hermitian-preserving for all $t\in\mathds{R}$.
Note that for $t\le 0$ we have $C_{\Phi_t}=I\otimes I+|t|d\,\omega\omega^*\geq 0$, hence $\Phi_t$ is completely positive
(and thus $\alpha$-positive for every $\alpha$). Therefore the only nontrivial regime is $t\ge 0$, which we assume below.

Let $\{e_i\}_{i=1}^d$ be the standard basis of $\mathds{C}^d$ and define
\[
\omega:=\frac{1}{\sqrt d}\sum_{i=1}^d e_i\otimes e_i\ \in\ \mathds{C}^d\otimes\mathds{C}^d.
\]
It is known that the Choi matrix of   $\Phi_t$ is $C_{\Phi_t}=I_d\otimes I_d-td\,\omega\omega^\ast.$ This is  a standard computation in the Choi/Jamio{\l}kowski representation; see, e.g.,
\cite[Sec.~2.2.2]{WatrousTQI} (and the original paper \cite{Choi75}).

\begin{lemma}
\label{lem:omega-overlap-max}
Let $x\in\mathds{C}^d\otimes\mathds{C}^d$ and write $X=\mathrm{mat}(x)\in \mathbb{M}_d$, so that $x=\mathrm{vec}(X)$.
Then for all unitaries $U,V\in \mathbb{M}_d$,
\begin{equation}
\label{eq:omega-overlap-trace}
\langle \omega,(U\otimes V)x\rangle = \frac{1}{\sqrt d}\,\Tr\!\big(V X U^{\mathsf T}\big).
\end{equation}
Consequently,
\begin{equation}
\label{eq:omega-overlap-max}
\max_{U,V\ \mathrm{unitary}} |\langle \omega,(U\otimes V)x\rangle|
= \frac{1}{\sqrt d}\,\|X\|_1,
\end{equation}
where $\|\cdot\|_1$ denotes the Schatten $1$-norm (trace norm).
\end{lemma}

\begin{proof}
Using \eqref{eq:vec-AXB} with $A=V$ and $B=U^{\mathsf T}$ gives
\[
(U\otimes V)\mathrm{vec}(X)=\mathrm{vec}(V X U^{\mathsf T}).
\]
Also $\omega=\frac{1}{\sqrt d}\mathrm{vec}(I_d)$ and
$$\langle \mathrm{vec}(A),\mathrm{vec}(B)\rangle=\mathrm{vec}(A)^*\mathrm{vec}(B)=\sum_{i,j}\overline{a_{ij}}b_{ij}=\Tr(A^\ast B).$$
Thus
\[
\langle \omega,(U\otimes V)x\rangle
=\frac{1}{\sqrt d}\,\langle \mathrm{vec}(I_d),\mathrm{vec}(V X U^{\mathsf T})\rangle
=\frac{1}{\sqrt d}\,\Tr(V X U^{\mathsf T}),
\]
proving \eqref{eq:omega-overlap-trace}. Finally, as $U$ ranges over unitaries so does $U^{\mathsf T}$, and the trace norm duality
(von Neumann trace inequality) yields
\[
\max_{U,V\ \mathrm{unitary}}|\Tr(V X U^{\mathsf T})|=\|X\|_1,
\]
giving \eqref{eq:omega-overlap-max}.
\end{proof}

\begin{lemma}
\label{lem:max-trace-norm}
Fix $d\ge 2$, $k\in\{1,\dots,d-1\}$, and $\theta\in[0,1)$.
Let $\sigma_1\ge\cdots\ge\sigma_d\ge 0$ satisfy
\[
\sigma_{k+2}=\cdots=\sigma_d=0,\qquad \sum_{j=1}^{k+1}\sigma_j^2=1.
\]
If $\theta=0$, assume $\sigma_{k+1}=0$ (so $\sum_{j=1}^{k}\sigma_j^2=1$).
If $\theta>0$, assume
\begin{equation}
\label{eq:ratio-constraint-sigma}
\sigma_{k+1}\ \le\ \frac{\theta}{k}\sum_{j=1}^k \sigma_j.
\end{equation}
Then
\begin{equation}
\label{eq:max-sum-sigma}
\sum_{j=1}^{k+1}\sigma_j\ \le\ \frac{k+\theta}{\sqrt{k+\theta^2}}.
\end{equation}
Moreover, equality holds if and only if
\[
\sigma_1=\cdots=\sigma_k=\frac{1}{\sqrt{k+\theta^2}}
\qquad\text{and}\qquad
\sigma_{k+1}=\frac{\theta}{\sqrt{k+\theta^2}}
\]
(with the convention $\sigma_{k+1}=0$ when $\theta=0$).
\end{lemma}
\begin{proof}
Set
\[
S:=\sum_{j=1}^k \sigma_j,\qquad b:=\sigma_{k+1}.
\]
Then $\sum_{j=1}^{k}\sigma_j^2 = 1-b^2$ by the normalization
$\sum_{j=1}^{k+1}\sigma_j^2=1$.

\medskip
\noindent\emph{Step 1: Cauchy--Schwarz bound for $S$.}
By the Cauchy--Schwarz inequality we have
\begin{equation}\label{eq:CS-detailed}
S^2=\Big(\sum_{j=1}^k \sigma_j\Big)^2
\le k\sum_{j=1}^k \sigma_j^2
= k(1-b^2).
\end{equation}
Hence
\begin{equation}\label{eq:S-CS}
S \le \sqrt{k(1-b^2)}.
\end{equation}
Moreover, equality holds in \eqref{eq:CS-detailed} if and only if
\begin{equation}\label{eq:equal-sigmas}
\sigma_1=\cdots=\sigma_k.
\end{equation}

\medskip
\noindent\emph{Step 2: the case $\theta=0$.}
In this case the constraint \eqref{eq:ratio-constraint-sigma} becomes $\sigma_{k+1}\le 0$, hence $b=\sigma_{k+1}=0$.
Then \eqref{eq:S-CS} gives $S\le \sqrt{k}$ and therefore
\[
\sum_{j=1}^{k+1}\sigma_j = S \le \sqrt{k}
= \frac{k}{\sqrt{k}}
=\frac{k+\theta}{\sqrt{k+\theta^2}},
\]
which is \eqref{eq:max-sum-sigma}.
Equality holds if and only if \eqref{eq:equal-sigmas} holds and
$\sum_{j=1}^k \sigma_j^2=1$, i.e. $\sigma_1=\cdots=\sigma_k=1/\sqrt{k}$ and $b=0$.

\medskip
\noindent\emph{Step 3: assume $\theta\in(0,1)$.}
The ratio constraint \eqref{eq:ratio-constraint-sigma} gives $b \le \frac{\theta}{k}S$.
Combining with \eqref{eq:S-CS} yields
\[
b \le \frac{\theta}{k}\sqrt{k(1-b^2)}
=\frac{\theta}{\sqrt{k}}\sqrt{1-b^2}.
\]
Since both sides are nonnegative, we may square to obtain
\[
b^2 \le \frac{\theta^2}{k}(1-b^2).
\]
Rearranging gives
\[
\Big(1+\frac{\theta^2}{k}\Big)b^2 \le \frac{\theta^2}{k}
\qquad\Longrightarrow\qquad
b^2 \le \frac{\theta^2}{k+\theta^2}.
\]
Equivalently,
\begin{equation}\label{eq:bmax}
0\le b \le b_{\max}:=\frac{\theta}{\sqrt{k+\theta^2}}.
\end{equation}

\medskip
\noindent\emph{Step 4: reduce to a one-variable maximization.}
We want to bound
\[
\sum_{j=1}^{k+1}\sigma_j = S+b.
\]
For each fixed $b$, the quantity $S=\sum_{j=1}^k \sigma_j$ is maximized under the constraint
$\sum_{j=1}^k \sigma_j^2 = 1-b^2$ exactly when $\sigma_1=\cdots=\sigma_k$
(by equality in Cauchy--Schwarz), and the maximal value is $\sqrt{k(1-b^2)}$.
Therefore, for every admissible choice of $(\sigma_1,\dots,\sigma_{k+1})$,
\begin{equation}\label{eq:reduce-f}
S+b \le \sqrt{k(1-b^2)} + b =: f(b).
\end{equation}
Moreover, by \eqref{eq:bmax}, we only need to consider $b\in[0,b_{\max}]$.

\medskip
\noindent\emph{Step 5: show $f$ is increasing on $[0,b_{\max}]$.}
Compute, for $b\in(0,1)$,
\[
f'(b)=\frac{d}{db}\big(\sqrt{k(1-b^2)}+b\big)
= -\frac{\sqrt{k}\,b}{\sqrt{1-b^2}} + 1
= 1-\frac{\sqrt{k}\,b}{\sqrt{1-b^2}}.
\]
Thus $f'(b)\ge 0$ is equivalent to
\[
\frac{\sqrt{k}\,b}{\sqrt{1-b^2}} \le 1
\quad\Longleftrightarrow\quad
k b^2 \le 1-b^2
\quad\Longleftrightarrow\quad
b^2 \le \frac{1}{k+1}.
\]
Now, since $\theta\in(0,1)$, we have $\theta^2\le 1$, hence
\[
b_{\max}^2=\frac{\theta^2}{k+\theta^2}\le \frac{1}{k+1}.
\]
(Actually $\theta^2\leq 1$ if and only if $\frac{\theta^2}{k+\theta^2}\le \frac{1}{k+1}$). Therefore $f'(b)\ge 0$ for all $b\in[0,b_{\max}]$, i.e. $f$ is increasing on this interval.

Consequently, the maximum of $f(b)$ on $[0,b_{\max}]$ is attained at $b=b_{\max}$.

\medskip
\noindent\emph{Step 6: evaluate $f(b_{\max})$.}
Using $b_{\max}=\theta/\sqrt{k+\theta^2}$, we compute
\[
1-b_{\max}^2 = 1-\frac{\theta^2}{k+\theta^2}=\frac{k}{k+\theta^2}.
\]
Hence
\[
\begin{aligned}
f(b_{\max})
&=\sqrt{k(1-b_{\max}^2)}+b_{\max}\\
&=\sqrt{k\cdot\frac{k}{k+\theta^2}}+\frac{\theta}{\sqrt{k+\theta^2}}\\
&=\frac{k+\theta}{\sqrt{k+\theta^2}}.
\end{aligned}
\]
Combining with \eqref{eq:reduce-f} gives
\[
\sum_{j=1}^{k+1}\sigma_j = S+b \le f(b) \le f(b_{\max})
=\frac{k+\theta}{\sqrt{k+\theta^2}},
\]
which is exactly \eqref{eq:max-sum-sigma}.

\medskip
\noindent\emph{Step 7: equality conditions.}
Assume equality holds in \eqref{eq:max-sum-sigma}. Tracing back through the proof,
we must have equality in each of the following:
\begin{enumerate}
\item equality in Cauchy--Schwarz \eqref{eq:CS-detailed}, hence $\sigma_1=\cdots=\sigma_k=:a$;
\item tightness of the ratio constraint, i.e. $b=(\theta/k)S$;
\item attainment at the endpoint $b=b_{\max}$ (equivalently, the bound \eqref{eq:bmax} is tight).
\end{enumerate}
With $\sigma_1=\cdots=\sigma_k=a$, we have $S=ka$ and the tight ratio constraint gives $b=\theta a$.
The normalization $\sum_{j=1}^{k+1}\sigma_j^2=1$ becomes
\[
ka^2 + (\theta a)^2 = (k+\theta^2)a^2 = 1,
\]
so $a=1/\sqrt{k+\theta^2}$ and $b=\theta/\sqrt{k+\theta^2}$, i.e.
\[
\sigma_1=\cdots=\sigma_k=\frac{1}{\sqrt{k+\theta^2}},
\qquad
\sigma_{k+1}=\frac{\theta}{\sqrt{k+\theta^2}}.
\]
Conversely, these values satisfy all hypotheses and make every inequality above an equality,
so they indeed achieve equality in \eqref{eq:max-sum-sigma}.
\end{proof}


\begin{theorem}[Sharp $\alpha$-positivity threshold for $\Phi_t$]
\label{thm:alpha-threshold-Phi-t}
Let $d\ge 2$ and $\Phi_t:\mathbb{M}_d\to \mathbb{M}_d$ be given by \eqref{eq:Phi-t}.
Fix $\alpha\in[1,d]$ and write $k:=\lfloor\alpha\rfloor$ and $\theta:=\alpha-k\in[0,1)$.
Then
\begin{equation}
\label{eq:threshold}
\Phi_t\in \mathsf{P}_\alpha
\quad\Longleftrightarrow\quad
t\ \le\ t_\alpha^\ast:=\frac{k+\theta^2}{(k+\theta)^2}.
\end{equation}
Equivalently,
\begin{equation}
\label{eq:lambda-explicit}
\lambda_\alpha(C_{\Phi_t}) \;=\; 1 - t\,\frac{(k+\theta)^2}{k+\theta^2},
\end{equation}
so $\Phi_t$ is $\alpha$-positive if and only if $\lambda_\alpha(C_{\Phi_t})\ge 0$.

Moreover, the threshold is sharp: if $t>t_\alpha^\ast$, then there exists $x\in\mathcal{V}_\alpha$ such that
$\langle x,C_{\Phi_t}x\rangle<0$.
If $\alpha\in(k,k+1)$, one may take $x$ with Schmidt coefficients
\[
\underbrace{\frac{1}{\sqrt{k+\theta^2}},\dots,\frac{1}{\sqrt{k+\theta^2}}}_{k\ \text{times}},
\qquad
\frac{\theta}{\sqrt{k+\theta^2}},
\qquad
0,\dots,0.
\]
If $\alpha=d$, one may take $x=\omega$.
\end{theorem}
\begin{proof}
Let $\omega=\frac{1}{\sqrt d}\mathrm{vec}(I_d).$  As   $C_{\Phi_t}=I_d\otimes I_d-td\,\omega\omega^\ast$,  for any unit vector $x$ we have
\begin{equation}\label{eq:test-again}
\langle x,\,C_{\Phi_t}x\rangle \;=\; 1-td\,|\langle \omega,x\rangle|^2.
\end{equation}
Recalling the definition \eqref{eq:lambda-alpha},
\[
\lambda_\alpha(C_{\Phi_t})
=\min_{x\in\mathcal{V}_\alpha}\langle x,C_{\Phi_t}x\rangle,
\]
we obtain from \eqref{eq:test-again} that
\begin{align}
\lambda_\alpha(C_{\Phi_t})
=\min_{x\in\mathcal{V}_\alpha}\big(1-td\,|\langle \omega,x\rangle|^2\big)=1-td\;\max_{x\in\mathcal{V}_\alpha}|\langle \omega,x\rangle|^2=1-td\,m_\alpha,
\label{eq:lambda-reduction}
\end{align}
where we set
\[
m_\alpha \;:=\;\max_{x\in\mathcal{V}_\alpha}|\langle \omega,x\rangle|^2.
\]
Thus it remains to compute $m_\alpha$.

\smallskip
\noindent\emph{Case 1: $\alpha=d$.}
Here $\mathcal{V}_d$ is the full unit sphere (since $\theta=0$ and the constraints in
Definition~\ref{def:alpha-admissible} are vacuous). By the Cauchy--Schwarz inequality we have
$|\langle \omega,x\rangle|\le \|\omega\|\,\|x\|=1$ for all unit $x$, and equality is attained at $x=\omega$.
Hence $m_d=1$ and therefore
\[
\lambda_d(C_{\Phi_t})=1-td.
\]
Thus $\lambda_d(C_{\Phi_t})\ge 0$ if and only if $t\le 1/d$, i.e.\ \eqref{eq:threshold} holds for $\alpha=d$.
Sharpness is immediate: if $t>1/d$, then $\langle \omega,C_{\Phi_t}\omega\rangle=1-td<0$.

\smallskip
\noindent\emph{Case 2: $\alpha\in[1,d)$ (so $k:=\lfloor\alpha\rfloor\in\{1,\dots,d-1\}$ and $\theta:=\alpha-k\in[0,1)$).}
We compute $m_\alpha$.

\smallskip
\noindent\emph{Step 1: reduce the maximization to a trace-norm problem.}
Fix $x\in\mathds{C}^d\otimes\mathds{C}^d$ and write $X=\mathrm{mat}(x)\in \mathbb{M}_d$ so that $x=\mathrm{vec}(X)$.
By Lemma~\ref{lem:omega-overlap-max},
\[
\max_{U,V\ \mathrm{unitary}}|\langle \omega,(U\otimes V)x\rangle|
=\frac{1}{\sqrt d}\,\|X\|_1.
\]
Now note that $\mathcal{V}_\alpha$ is invariant under local unitaries:
if $x\in\mathcal{V}_\alpha$ and $U,V$ are unitaries, then $(U\otimes V)x$ has the same Schmidt coefficients as $x$,
hence $(U\otimes V)x\in\mathcal{V}_\alpha$ (by Definition~\ref{def:alpha-admissible}).
Therefore,
\begin{align*}
\max_{y\in\mathcal{V}_\alpha}|\langle \omega,y\rangle|
&=\max_{x\in\mathcal{V}_\alpha}\;\max_{U,V\ \mathrm{unitary}}|\langle \omega,(U\otimes V)x\rangle|\\
&=\max_{x\in\mathcal{V}_\alpha}\;\frac{1}{\sqrt d}\,\|\mathrm{mat}(x)\|_1,
\end{align*}
and consequently
\begin{equation}\label{eq:Malpha-tracenorm}
m_\alpha
=\frac{1}{d}\Big(\max_{x\in\mathcal{V}_\alpha}\|\mathrm{mat}(x)\|_1\Big)^2.
\end{equation}

\smallskip
\noindent\emph{Step 2: express the constraint $x\in\mathcal{V}_\alpha$ in terms of singular values.}
Since $\mathrm{mat}(\cdot)$ is an isometry between the Hilbert space norm on $\mathds{C}^d\otimes\mathds{C}^d$
and the Frobenius norm on $\mathbb{M}_d$ (\eqref{Frobs}), we have
\[
\|x\|=1 \quad\Longleftrightarrow\quad \|X\|_2=1,
\]
where $\|\cdot\|_2$ is the Frobenius norm.
Moreover, the Schmidt coefficients of $x$ coincide with the singular values of $X$:
\[
s_j(x)=\sigma_j(X)\qquad (j=1,\dots,d).
\]
Hence the conditions $x\in\mathcal{V}_\alpha$ in Definition~\ref{def:alpha-admissible} translate exactly into:
\[
\sigma_{k+2}(X)=\cdots=\sigma_d(X)=0,\qquad \sum_{j=1}^{k+1}\sigma_j(X)^2=\|X\|_2^2=1,
\]
and, when $\theta>0$,
\[
\sigma_{k+1}(X)\le \frac{\theta}{k}\sum_{j=1}^k \sigma_j(X).
\]
Finally, $\|X\|_1=\sum_{j=1}^d \sigma_j(X)=\sum_{j=1}^{k+1}\sigma_j(X)$ under the rank ceiling.

\smallskip
\noindent\emph{Step 3: maximize $\|X\|_1$ under the fractional constraint.}
By Lemma~\ref{lem:max-trace-norm},
\[
\|X\|_1=\sum_{j=1}^{k+1}\sigma_j(X)\ \le\ \frac{k+\theta}{\sqrt{k+\theta^2}}.
\]
Therefore \eqref{eq:Malpha-tracenorm} yields
\[
m_\alpha
=\frac{1}{d}\left(\frac{k+\theta}{\sqrt{k+\theta^2}}\right)^2
=\frac{(k+\theta)^2}{d\,(k+\theta^2)}.
\]
Plugging this into $\lambda_\alpha(C_{\Phi_t})=1-td\,m_\alpha$ gives
\[
\lambda_\alpha(C_{\Phi_t})
=1-td\cdot \frac{(k+\theta)^2}{d\,(k+\theta^2)}
=1-t\,\frac{(k+\theta)^2}{k+\theta^2},
\]
which is \eqref{eq:lambda-explicit}. In particular,
\[
\lambda_\alpha(C_{\Phi_t})\ge 0
\quad\Longleftrightarrow\quad
t\le \frac{k+\theta^2}{(k+\theta)^2}
=:t_\alpha^\ast,
\]
proving \eqref{eq:threshold} for $\alpha\in[1,d)$ as well.

\smallskip
\noindent\emph{Sharpness.}
Assume first that $\alpha\in(k,k+1)$ (so $\theta\in(0,1)$).
Let $\sigma_1=\cdots=\sigma_k=\frac{1}{\sqrt{k+\theta^2}}$ and
$\sigma_{k+1}=\frac{\theta}{\sqrt{k+\theta^2}}$, and set
\[
X:=\mathrm{diag}(\sigma_1,\dots,\sigma_{k+1},0,\dots,0)\in \mathbb{M}_d,
\qquad x:=\mathrm{vec}(X).
\]
Then $\|x\|=\|X\|_2=1$, the singular values of $X$ are exactly the above $\sigma_j$,
and the ratio constraint holds with equality, hence $x\in\mathcal{V}_\alpha$.
Moreover, for this diagonal positive $X$ we have $\|X\|_1=\Tr(X)=\sum_{j=1}^{k+1}\sigma_j$ and
\[
\langle \omega,x\rangle
=\frac{1}{\sqrt d}\Tr(X)
=\frac{1}{\sqrt d}\|X\|_1,
\]
so $|\langle \omega,x\rangle|^2=m_\alpha$ (indeed, $\|X\|_1$ attains the maximum in Lemma~\ref{lem:max-trace-norm}).
Hence, by \eqref{eq:test-again},
\[
\langle x,C_{\Phi_t}x\rangle = 1-td\,m_\alpha<0
\qquad\text{whenever } t>t_\alpha^\ast.
\]
This proves sharpness for $\alpha\in(k,k+1)$. For $\alpha=d$ sharpness was shown in Case~1
(using $x=\omega$). If $\theta=0$ (i.e.\ $\alpha=k<d$), take
$X=\mathrm{diag}(1/\sqrt{k},\dots,1/\sqrt{k},0,\dots,0)$ and $x=\mathrm{vec}(X)$.
Then $x\in\mathcal V_k$, $\|X\|_1=\sqrt{k}$, so $|\langle\omega,x\rangle|^2 = k/d = m_k$ and thus
$\langle x, C_{\Phi_t}x\rangle = 1 - td\,m_k = 1-tk <0$ whenever $t>1/k$.

\end{proof}

\subsection{\texorpdfstring{A fractional Kraus characterization of $\alpha$-superpositive maps}{A fractional Kraus characterization of alpha-superpositive maps}}
\label{subsec:fractional-kraus}

A central benefit of the fractional framework is that the dual cone on maps admits an intrinsic Kraus-type
description.  In the integer case, $k$-superpositivity is equivalent to the existence of a Kraus decomposition with
Kraus operators of rank at most $k$.  The main result of this subsection shows that the fractional level $\alpha$ admits a precise
analog, in which the classical rank bound is replaced by the same singular-value constraint underlying
$\alpha$-admissibility.

First note that we can define  the $\alpha$-admissibility for a matrix $X\in \mathbb{M}_{m,n}$ as well. Since the singular values of $X=\mathrm{mat} (x)$ are exactly the Schmidt coefficients of $x=\mathrm{Vec}(X)$.

\begin{definition}[$\alpha$-admissible matrices]
\label{def:AdmA}
Fix $\alpha\in[1,d]$, write $k:=\lfloor\alpha\rfloor$, $\theta:=\alpha-k\in[0,1)$, and $r:=\lceil\alpha\rceil$.
A matrix $A\in \mathbb{M}_{m,n}$ is called \emph{$\alpha$-admissible} if
\begin{enumerate}
\item[(i)] (\emph{Rank ceiling}) $\sigma_j(A)=0$ for all $j\ge r+1$ (equivalently $\mathrm{rank}(A)\le r$);
\item[(ii)] (\emph{Normalized ratio}) if $\theta>0$, then
\begin{align}\label{eq:A-adm}
\sigma_{k+1}(A)\ \le\ \frac{\theta}{k}\sum_{j=1}^k\sigma_j(A).
\end{align}
\end{enumerate}
We denote the set of all $\alpha$-admissible matrices by $\mathfrak{A}_\alpha$.
\end{definition}

\begin{remark}
\label{rem:normratio}
If $A\neq 0$ and $\mathrm{rank}(A)\le k+1$, then \eqref{eq:A-adm} is equivalent to
\[
\frac{\|A\|_1}{\|A\|_{(k)}}\le 1+\frac{\theta}{k},
\]
where $\|A\|_1=\sum_{j=1}^{\min(m,n)}\sigma_j(A)$ is the Schatten $1$-norm (nuclear norm) and
$\|A\|_{(k)}=\sum_{j=1}^k\sigma_j(A)$ is the Ky Fan $k$-norm.
(For $A=0$, \eqref{eq:A-adm} holds trivially.)
\end{remark}

\medskip

\begin{theorem}[Fractional Kraus characterization of $\alpha$-superpositive maps]
\label{thm:fractional-kraus}
Fix $\alpha\in[1,d]$ and write $\alpha=k+\theta$ with $k=\lfloor\alpha\rfloor$ and $\theta\in[0,1)$.
Let $\Phi:\mathbb{M}_n\to \mathbb{M}_m$ be linear. The following are equivalent:
\begin{enumerate}
\item[\textup{(i)}] $\Phi\in\mathsf{SP}_\alpha$, i.e.\ $C_\Phi\in\mathsf{K}_\alpha$.
\item[\textup{(ii)}] There exist $N\in\mathds{N}$ and matrices $A_1,\dots,A_N\in\mathfrak{A}_\alpha$ such that
\begin{equation}
\label{eq:alpha-kraus}
\Phi(X)=\sum_{i=1}^N A_i X A_i^\ast\qquad\text{for all }X\in \mathbb{M}_n.
\end{equation}
\item[\textup{(iii)}] There exist $N\in\mathds{N}$, vectors $\psi_1,\dots,\psi_N\in\mathcal{V}_\alpha$, and scalars $t_1,\dots,t_N>0$ such that
\begin{equation}
\label{eq:alpha-choi-decomp}
C_\Phi=\sum_{i=1}^N t_i\,\psi_i \psi_i^\ast.
\end{equation}
\end{enumerate}
Moreover, in \textup{(iii)} one may choose $N\le (nm)^2$.

In particular, $\mathsf{SP}_\alpha$ consists precisely of the completely positive maps generated by Kraus operators whose singular values
satisfy the fractional constraint \eqref{eq:A-adm}.
\end{theorem}

\begin{proof}
\textup{(iii)$\Rightarrow$(i)} is immediate from the definition of $\mathsf{K}_\alpha$.

\smallskip
\noindent\textup{(ii)$\Rightarrow$(iii)} Assume \eqref{eq:alpha-kraus} with $A_i\in\mathfrak{A}_\alpha$.
It is an standard use of the Choi--Jamio\l kowski correspondence that
\[
C_\Phi=\sum_{i=1}^N \mathrm{vec}(A_i)\,\mathrm{vec}(A_i)^\ast.
\]
For each $i$ with $A_i\neq 0$, set $\psi_i:=\mathrm{vec}(A_i)/\|\mathrm{vec}(A_i)\|$ and $t_i:=\|\mathrm{vec}(A_i)\|^2$.
Since $A_i\in\mathfrak{A}_\alpha$ we have $\psi_i\in\mathcal{V}_\alpha$. This yields \eqref{eq:alpha-choi-decomp}.

\smallskip
\noindent\textup{(i)$\Rightarrow$(iii)} Let $\mathcal{S}_\alpha:=\{\psi\psi^\ast:\ \psi\in\mathcal{V}_\alpha\}$.
By Corollary~\ref{cor:Kalpha-no-closure}, $\mathsf{K}_\alpha=\mathrm{cone}(\mathcal{S}_\alpha)$.
Thus $C_\Phi\in\mathsf{K}_\alpha$ means $C_\Phi=t\rho$ for some $t:=\Tr(C_\Phi)\ge 0$ and some
\[
\rho:=\frac{1}{\Tr(C_\Phi)}\,C_\Phi \ \in\ \mathrm{conv}(\mathcal{S}_\alpha),
\]
where we interpret $\rho$ arbitrarily if $C_\Phi=0$.
Since every element of $\mathcal{S}_\alpha$ has trace $1$, the convex set $\mathrm{conv}(\mathcal{S}_\alpha)$ lies in the affine hyperplane
\[
\mathcal{H}:=\{H\in (\mathbb{M}_n\otimes \mathbb{M}_m)^{\mathrm{h}}:\ \Tr(H)=1\},
\]
which has real dimension $(nm)^2-1$. By Carath\'eodory's theorem in $\mathcal{H}$
\cite[Thm.~1.3]{BarvinokConvexity} (see also \cite[Sec.~17]{RockafellarCA}),
we may write
\[
\rho=\sum_{i=1}^{N} p_i\, \psi_i \psi_i^\ast
\quad\text{with}\quad
N\le (nm)^2,\ \ p_i\ge 0,\ \sum_i p_i=1,\ \ \psi_i\in\mathcal{V}_\alpha.
\]
Therefore
\[
C_\Phi=\sum_{i=1}^N (t p_i)\, \psi_i \psi_i^\ast,
\]
which is \eqref{eq:alpha-choi-decomp} after removing zero coefficients.

\smallskip
\noindent\textup{(iii)$\Rightarrow$(ii)} Assume \eqref{eq:alpha-choi-decomp}.
For each $i$   choose $A_i\in\mathbb M_{m,n}$ such that  $\mathrm{vec}(A_i)=\sqrt{t_i}\,\psi_i$ and   $A_i\in\mathfrak{A}_\alpha$ by $\psi_i\in\mathcal{V}_\alpha$.
Define $\Psi:\mathbb{M}_n\to \mathbb{M}_m$ by $\Psi(X)=\sum_{i=1}^N A_i X A_i^\ast$.
Computing the Choi matrix of $\Psi$ we have
\[
C_\Psi=\sum_{i=1}^N \mathrm{vec}(A_i)\,\mathrm{vec}(A_i)^\ast
=\sum_{i=1}^N t_i\,\psi_i \psi_i^\ast
=C_\Phi.
\]
Since $\Phi\mapsto C_\Phi$ is injective,
we conclude that $\Psi=\Phi$, hence \eqref{eq:alpha-kraus} holds.

\end{proof}

\begin{remark}[Integer recovery and the endpoint $\alpha=d$]
If $\alpha=k$ is an integer, then $\theta=0$ and \eqref{eq:A-adm} reduces to $\sigma_{k+1}(A)=0$, i.e.\ $\mathrm{rank}(A)\le k$.
Thus Theorem~\ref{thm:fractional-kraus} recovers the classical characterization of $k$-superpositive maps.

If $\alpha=d$ (hence $k=d$ and $\theta=0$), then $\mathfrak{A}_d=\mathbb{M}_{m,n}$ and $\mathsf{SP}_d$ coincides with the full cone of
completely positive maps (equivalently $C_\Phi\ge 0$).
\end{remark}

 \section{Explicit Examples and Thresholds}\label{sec:applications}
\noindent
This section illustrates how the fractional cones lead to concrete, verifiable criteria in symmetric situations.
We begin by introducing the critical indices $\FSN(\rho)$ and $\tau(\Phi)$ and the extremal quadratic forms
$\lambda_\alpha(\cdot),\mu_\alpha(\cdot)$ which serve as the main computational interface with the abstract cone theory.
We then analyze the isotropic two-dimensional slice $\mathrm{span}\{I,P_\omega\}$, where membership in $\mathsf K_\alpha$
reduces to explicit linear constraints and yields closed-form formulas.
Finally, we relate these slice computations to the canonical depolarizing ray
$\Phi_t(X)=\Tr(X)I_d-tX$ through its Choi matrix and deduce the corresponding stability-index behavior without repeating
the full threshold proof.

\subsection{Critical indices and extremal quadratic forms}
\label{subsec:app-toolbox}

The fractional filtrations $(\mathsf K_\alpha)_{\alpha\in[1,d]}$ and $(\mathsf P_\alpha)_{\alpha\in[1,d]}$
induce natural \emph{critical indices} which convert the parameter $\alpha$ into numerical invariants.
These indices are the main interfaces between the abstract cone theory and concrete computations.

\begin{definition}[Fractional Schmidt index]
\label{def:FSN}
For a nonzero $\rho\in(\mathbb{M}_n\otimes\mathbb{M}_m)^+$, with $d:=\min\{n,m\}$, define
\begin{equation}
\label{eq:FSN-def}
\FSN(\rho)\;:=\;\inf\big\{\alpha\in[1,d]:\ \rho\in\mathsf{K}_\alpha\big\}.
\end{equation}
\end{definition}

\begin{definition}[Stability index]
\label{def:tau}
For a Hermitian-preserving map $\Phi\in\mathsf P_1$   with $d:=\min\{n,m\}$, define
\begin{equation}
\label{eq:tau-def}
\tau(\Phi)\;:=\;\sup\big\{\alpha\in[1,d]:\ \Phi\in\mathsf{P}_\alpha\big\}.
\end{equation}
This gives the largest fractional level up to which $\Phi$ remains $\alpha$-positive.
\end{definition}

\begin{remark}
Both indices are scale-invariant in the natural sense: $\FSN(c\rho)=\FSN(\rho)$,  and
$\tau(c\Phi)=\tau(\Phi)$ for all $c>0$ (since each $\mathsf{P}_\alpha$ is a cone).
\end{remark}

\medskip
\noindent
We will repeatedly use extremal quadratic forms over $\mathcal{V}_\alpha$.
Recall (see \eqref{eq:lambda-alpha} and Proposition~\ref{prop:compact-lambda}) that for
$W\in(\mathbb{M}_n\otimes\mathbb{M}_m)^{\mathrm{h}}$,
\[
\lambda_\alpha(W):=\min\{\langle \psi,W\psi\rangle:\ \psi\in\mathcal{V}_\alpha\},
\qquad
W\in\mathsf{BP}_\alpha\iff \lambda_\alpha(W)\ge 0,
\]
and we also set
\begin{equation}
\label{eq:mu-alpha-def}
\mu_\alpha(W):=\max\{\langle \psi,W\psi\rangle:\ \psi\in\mathcal{V}_\alpha\}
= -\lambda_\alpha(-W),
\end{equation}
which exists since $\mathcal{V}_\alpha$ is compact.

\begin{proposition}[Recovery of the integer hierarchy]
\label{prop:integer-recovery}
Let $k\in\{1,\dots,d\}$.
\begin{enumerate}
\item[\textup{(i)}] For $\rho\in(\mathbb{M}_n\otimes\mathbb{M}_m)^+$,
\[
\rho\in\mathsf{K}_k \quad\Longleftrightarrow\quad \FSN(\rho)\le k.
\]
Consequently,
\[
\min\big\{\ell\in\{1,\dots,d\}:\ \rho\in\mathsf{K}_\ell\big\}
=\big\lceil \FSN(\rho)\big\rceil.
\]

\item[\textup{(ii)}] For a Hermitian-preserving map $\Phi$,
\[
\Phi\in\mathsf{P}_k \quad\Longleftrightarrow\quad \tau(\Phi)\ge k.
\]
\end{enumerate}
\end{proposition}

\begin{proof}
Both statements follow immediately from the monotonicity
$\mathsf{K}_1\subseteq\cdots\subseteq \mathsf{K}_\alpha\subseteq\cdots\subseteq \mathsf{K}_d$
and
$\mathsf{P}_1\supseteq\cdots\supseteq \mathsf{P}_\alpha\supseteq\cdots\supseteq \mathsf{P}_d$,
together with the definitions \eqref{eq:FSN-def}--\eqref{eq:tau-def}.
\end{proof}


\subsection{An exact fractional profile on a rank-one perturbation slice}
\label{subsec:flagship-slice}

We now present a closed-form application showing that the fractional parameter $\alpha$
resolves the classical stepwise thresholds into a continuous profile.
The computation is completely finite-dimensional and takes place inside the two-dimensional real subspace
$\mathrm{span}\{I,P_\omega\}\subset\mathbb{M}_{d^2}$ generated by the identity and a rank-one projection.

\subsubsection*{Setup: a two-dimensional slice on $\mathds{C}^d\otimes\mathds{C}^d$}
Fix $d\ge 2$ and let $\{e_i\}_{i=1}^d$ be the standard basis of $\mathds{C}^d$.
Define
\begin{equation}
\label{eq:omega-def}
\omega \;:=\; \frac{1}{\sqrt d}\sum_{i=1}^d e_i\otimes e_i,
\qquad
P_\omega \;:=\; \omega\omega^\ast \in \mathbb{M}_{d^2}.
\end{equation}
Consider the normalized positive semidefinite family
\begin{equation}
\label{eq:rhoF-def}
\rho_F \;:=\; \frac{1-F}{d^2-1}\,(I-P_\omega)\;+\;F\,P_\omega,
\qquad 0\le F\le 1,
\end{equation}
so that $\Tr(\rho_F)=1$ and $\Tr(P_\omega\rho_F)=F$.

At the \emph{integer} levels $\alpha=k$, the family \eqref{eq:rhoF-def} exhibits a sharp step threshold on the slice
$\mathrm{span}\{I,P_\omega\}$: by the classical computation of Terhal--Horodecki
(see \cite[Theorem~2]{TerhalHorodeckiSN}), one has
\begin{equation}
\label{eq:integer-threshold-rhoF}
\rho_F\in \mathsf{K}_k
\qquad\Longleftrightarrow\qquad
F\le \frac{k}{d},
\qquad k\in\{1,\dots,d\}.
\end{equation}
Thus, along this rank-one perturbation slice, the integer hierarchy $\{\mathsf{K}_k\}_{k=1}^d$ yields a
stepwise membership rule with breakpoints at $F=k/d$.
Our goal is to compute the \emph{fractional} refinement of \eqref{eq:integer-threshold-rhoF}:
namely, to determine an explicit profile $f_d(\alpha)$ such that for every $\alpha\in[1,d]$,
\[
\rho_F\in \mathsf{K}_\alpha \quad\Longleftrightarrow\quad F\le f_d(\alpha),
\]
with $f_d(k)=k/d$ at integers and a continuous interpolation between successive integer thresholds.

\begin{lemma}
\label{lem:mu-alpha-Pomega}
Fix $d\ge 2$ and let $\omega\in\mathds{C}^d\otimes\mathds{C}^d$ and $P_\omega=\omega\omega^*$ be as in \eqref{eq:omega-def}.
Fix $\alpha\in[1,d]$ and set $k:=\lfloor\alpha\rfloor$, $\theta:=\alpha-k\in[0,1)$, and $r:=\lceil\alpha\rceil$.
Then
\begin{equation}
\label{eq:mu-alpha-Pomega-correct}
\mu_\alpha(P_\omega)
\;=\;
\max_{\psi\in\mathcal{V}_\alpha}\langle \psi,\,P_\omega\psi\rangle
\;=\;
\frac{(k+\theta)^2}{d\,(k+\theta^2)}.
\end{equation}
Moreover, the maximum is attained by a vector $\psi\in\mathcal{V}_\alpha$ whose Schmidt coefficients satisfy
\[
s_1=\cdots=s_k=\frac{1}{\sqrt{k+\theta^2}},
\qquad
s_{k+1}=\frac{\theta}{\sqrt{k+\theta^2}},
\qquad
s_j=0\ \ (j\ge k+2),
\]
(with the conventions that $s_{k+1}=0$ when $\theta=0$, and that when $\alpha=d$ one may take $\psi=\omega$).
\end{lemma}

\begin{proof}
If $\alpha=d$, then  $k=d$, $\theta=0$, and $\mathcal{V}_d$ is the full unit sphere of $\mathds{C}^d\otimes\mathds{C}^d$.
Hence
\[
\mu_d(P_\omega)=\max_{\|\psi\|=1}\langle\psi,P_\omega\psi\rangle
=\|P_\omega\|=1,
\]
and the maximum is attained at $\psi=\omega$.
On the other hand, the right-hand side of \eqref{eq:mu-alpha-Pomega-correct} equals
$\frac{d^2}{d\cdot d}=1$, so \eqref{eq:mu-alpha-Pomega-correct} holds.

If $\alpha\in[1,d)$, then $k=\lfloor\alpha\rfloor\in\{1,\dots,d-1\}$ and $r=k+1$.
Let $\psi\in\mathcal{V}_\alpha$ be a unit vector. Since $\mathcal{V}_\alpha$ is invariant under local unitaries,
we may apply suitable unitaries $U,V\in\mathbb{M}_d$ so that $(U\otimes V)\psi$ is Schmidt-diagonal
in the basis defining $\omega$. Since $\langle\psi,P_\omega\psi\rangle$ is unchanged by applying the same local unitaries
to both $\psi$ and $\omega$ (and $\omega$ is itself invariant under $U\otimes\overline{U}$), we may assume without loss of generality that
\[
\psi=\sum_{j=1}^{r} s_j\, e_j\otimes e_j,
\qquad s_1\ge\cdots\ge s_{r}\ge 0,
\qquad \sum_{j=1}^{r}s_j^2=1,
\]
where $r=k+1$ and $s_j=0$ for $j\ge r+1$.
Then
\[
\langle \psi,P_\omega\psi\rangle=|\langle\omega,\psi\rangle|^2
=\frac{1}{d}\Big(\sum_{j=1}^{r} s_j\Big)^2
=\frac{1}{d}\Big(\sum_{j=1}^{k+1} s_j\Big)^2.
\]

Set $S:=\sum_{j=1}^k s_j$ and $b:=s_{k+1}$. If $\theta>0$, the defining constraint for $\mathcal{V}_\alpha$ gives
$b\le (\theta/k)S$. (If $\theta=0$, then $r=k$ and hence $b=0$; equivalently, the same inequality holds with $b=0$.)
Thus the problem reduces to maximizing $(S+b)^2$ subject to
\[
\sum_{j=1}^k s_j^2+b^2=1,
\qquad
b\le \frac{\theta}{k}S \ \ (\theta>0),
\qquad s_j\ge 0.
\]

This is exactly the scalar optimization carried out in Lemma~\ref{lem:max-trace-norm} (with $\sigma_j=s_j$):
it yields the sharp bound
\[
S+b=\sum_{j=1}^{k+1}s_j \;\le\; \frac{k+\theta}{\sqrt{k+\theta^2}},
\]
with equality if and only if $s_1=\cdots=s_k=1/\sqrt{k+\theta^2}$ and $b=\theta/\sqrt{k+\theta^2}$
(with $b=0$ when $\theta=0$).
Therefore
\[
\langle \psi,P_\omega\psi\rangle
=\frac{1}{d}(S+b)^2
\le \frac{1}{d}\left(\frac{k+\theta}{\sqrt{k+\theta^2}}\right)^2
=\frac{(k+\theta)^2}{d\,(k+\theta^2)}.
\]
This proves \eqref{eq:mu-alpha-Pomega-correct}. The equality conditions give the stated maximizers.
\end{proof}

\medskip

Define the profile function
\begin{equation}
\label{eq:fd-def}
f_d(\alpha)\;:=\;\mu_\alpha(P_\omega).
\end{equation}

\medskip
\noindent
To obtain closed-form constraints, we restrict to the isotropic two-dimensional subspace
$\mathrm{span}\{I,P_\omega\}$, where the quadratic form tests against $\mathcal V_\alpha$ reduce to a one-variable
optimization problem.  This symmetry reduction converts membership in $\mathsf K_\alpha$ into explicit linear
inequalities in the parameters of $aI+bP_\omega$.  The next theorem records the resulting sharp characterization.

\begin{theorem}[Exact fractional profile on the slice $\mathrm{span}\{I,P_\omega\}$]
\label{thm:flagship-slice}
Let $d\ge 2$, let $\rho_F$ be as in \eqref{eq:rhoF-def}, and let $\alpha\in[1,d]$.
Then
\begin{equation}
\label{eq:rhoF-membership}
\rho_F\in\mathsf{K}_\alpha
\qquad\Longleftrightarrow\qquad
F\le f_d(\alpha).
\end{equation}
Equivalently, writing $\alpha=k+\theta$ with $k=\lfloor\alpha\rfloor$ and $\theta\in[0,1)$, one has
\begin{equation}
\label{eq:fd-explicit}
f_d(\alpha)=\frac{(k+\theta)^2}{d\,(k+\theta^2)}.
\end{equation}
In particular, for integer $\alpha=k$ this gives $f_d(k)=k/d$, recovering the classical integer threshold on this slice
(cf.\ \cite[Theorem~2]{TerhalHorodeckiSN}).
\end{theorem}

\begin{proof}
By Lemma~\ref{lem:mu-alpha-Pomega}, the Hermitian operator
\begin{equation}
\label{eq:Walpha}
W_\alpha \;:=\; I-\frac{1}{f_d(\alpha)}\,P_\omega
\end{equation}
satisfies $\langle\psi,W_\alpha\psi\rangle\ge 0$ for all $\psi\in\mathcal{V}_\alpha$, hence
$W_\alpha\in\mathsf{BP}_\alpha$.
Using $\Tr(P_\omega\rho_F)=F$, we obtain
\begin{equation}
\label{eq:Walpha-rhoF}
\Tr(W_\alpha\rho_F)=1-\frac{F}{f_d(\alpha)}.
\end{equation}
If $F>f_d(\alpha)$ then \eqref{eq:Walpha-rhoF} is negative and hence $\rho_F\notin(\mathsf{BP}_\alpha)^\ast=\mathsf{K}_\alpha$.

Conversely, assume $F\le f_d(\alpha)$ and let $W\in\mathsf{BP}_\alpha$.

Let $\mathcal{T}: \mathbb{M}_{d^2}\to \mathbb{M}_{d^2}$ be the Haar-averaging map
\[
\mathcal{T}(X)\;:=\;\int_{\mathbb{U}(d)} (U\otimes \overline{U})^\ast\,X\,(U\otimes \overline{U})\,\mathrm{d}U,
\]
where $\mathrm{d}U$ denotes the normalized Haar probability measure on $\mathbb{U}(d)$.

Set $\overline{W}:=\mathcal{T}(W)$. Since $\mathcal{V}_\alpha$ is invariant under local unitaries,
we have $\overline{W}\in \mathsf{BP}_\alpha$.

Moreover, $\rho_F\in \mathrm{span}\{I,P_\omega\}$ is $(U\otimes \overline{U})$-invariant, hence
\[
\Tr(W\rho_F)=\Tr(\overline{W}\rho_F).
\]

Finally, it is a standard fact that the image of $\mathcal{T}$ is exactly the two-dimensional space
$\mathrm{span}\{I,P_\omega\}$ (equivalently, the commutant of $\{U\otimes\overline{U}:U\in\mathbb{U}(d)\}$),
so there exist $a,b\in\mathds{R}$ such that
\[
\overline{W}=aI+bP_\omega.
\]
See, e.g., \cite[Example~49]{MeleHaarTutorial}.
Now $\overline{W}\in\mathsf{BP}_\alpha$ means
\[
0\le \langle\psi,\overline{W}\psi\rangle = a + b\,\langle\psi,P_\omega\psi\rangle
\qquad\forall\,\psi\in\mathcal{V}_\alpha.
\]
Since $0\le \langle\psi,P_\omega\psi\rangle\le \mu_\alpha(P_\omega)=f_d(\alpha)$, this is equivalent to
$a\ge 0$ and $a+b f_d(\alpha)\ge 0$ (the minimum of $a+bt$ over $t\in[0,f_d(\alpha)]$ occurs at $t=0$ if $b\ge 0$
and at $t=f_d(\alpha)$ if $b\le 0$).
Therefore,
\[
\Tr(\overline{W}\rho_F)=a+bF\ge 0:
\]
indeed, if $b\ge 0$ then $a\ge 0$ gives $a+bF\ge 0$, while if $b\le 0$ then $F\le f_d(\alpha)$ gives
$a+bF\ge a+b f_d(\alpha)\ge 0$.
Hence $\Tr(W\rho_F)=\Tr(\overline{W}\rho_F)\ge 0$ for all $W\in\mathsf{BP}_\alpha$, so $\rho_F\in(\mathsf{BP}_\alpha)^\ast=\mathsf{K}_\alpha$.
\end{proof}

\medskip
\noindent
As an immediate consequence, one can invert the slice inequalities and express the fractional Schmidt index of
isotropic states directly in terms of the fidelity parameter.  This yields a closed-form ``fractional refinement''
of the classical step-function behavior of the Schmidt number on isotropic states.

 \begin{corollary}
\label{cor:FSN-rhoF}
Let $\rho_F$ be as in \eqref{eq:rhoF-def} and let $f_d(\alpha)$ be given by \eqref{eq:fd-def}--\eqref{eq:fd-explicit}.
Then:
\begin{enumerate}
\item[\textup{(i)}] If $0\le F\le \frac{1}{d}$, then $\FSN(\rho_F)=1$.
\item[\textup{(ii)}] If $\frac{1}{d}<F\le 1$, then $\FSN(\rho_F)$ is the unique $\alpha\in(1,d]$ such that
\begin{equation}
\label{eq:FSN-invert}
F=f_d(\alpha).
\end{equation}
More precisely, if $F\in\big[\frac{k}{d},\frac{k+1}{d}\big]$ with $k\in\{1,\dots,d-1\}$, then
\[
\FSN(\rho_F)=k+\theta(F),
\]
where $\theta(F)\in[0,1]$ is the unique solution of
\begin{equation}
\label{eq:theta-inverse-F-correct}
(Fd-1)\theta^2-2k\theta+k(Fd-k)=0.
\end{equation}
Equivalently, for $F>\frac{k}{d}$ (so that $Fd-1>0$), one may write
\begin{equation}
\label{eq:theta-explicit-F-correct}
\theta(F)=\frac{k-\sqrt{k^2-k(Fd-1)(Fd-k)}}{Fd-1}.
\end{equation}
At the endpoint $F=\frac{k}{d}$ one has $\theta(F)=0$.
\end{enumerate}
\end{corollary}

\begin{proof}
By Theorem~\ref{thm:flagship-slice}, $\rho_F\in\mathsf{K}_\alpha$ if and only if $F\le f_d(\alpha)$.
Since $\mathsf{K}_\alpha$ is increasing in $\alpha$, the definition \eqref{eq:FSN-def} gives
\[
\FSN(\rho_F)=\inf\{\alpha\in[1,d]:\ F\le f_d(\alpha)\}.
\]

\smallskip
\noindent\emph{Step 1: the regime $0\le F\le 1/d$.}
We have $f_d(1)=1/d$, hence $F\le f_d(1)$ and therefore $\rho_F\in\mathsf{K}_1$.
Since $\FSN(\rho_F)\ge 1$ by definition, it follows that $\FSN(\rho_F)=1$.

\smallskip
\noindent\emph{Step 2: the regime $F>1/d$.}
The function $f_d$ is continuous and strictly increasing on $[1,d]$ (indeed, on each interval
$\alpha\in[k,k+1]$ it is increasing in $\theta=\alpha-k$, with $f_d(k)=k/d$ and $f_d(k+1)=(k+1)/d$).
Hence for each $F\in(1/d,1]$ there is a unique $\alpha\in(1,d]$ such that $F=f_d(\alpha)$.
Moreover, by monotonicity of $f_d$, this unique $\alpha$ equals $\FSN(\rho_F)$.

Now fix $k\in\{1,\dots,d-1\}$ and assume $F\in[\frac{k}{d},\frac{k+1}{d}]$.
Write $\alpha=k+\theta$ with $\theta\in[0,1]$. Using \eqref{eq:fd-explicit}, the equation $F=f_d(\alpha)$ becomes
\[
Fd\,(k+\theta^2)=(k+\theta)^2,
\]
which expands to \eqref{eq:theta-inverse-F-correct}. Solving the quadratic yields \eqref{eq:theta-explicit-F-correct},
and the choice of the minus sign is the one giving $\theta(F)\in[0,1]$.
At $F=\frac{k}{d}$ the equation forces $\theta=0$, giving the endpoint value.
\end{proof}

\medskip

\subsubsection*{Compatibility with the map thresholds (without repetition)}

The same rank-one projection $P_\omega$ also governs the Choi matrices of the canonical ray
\begin{equation}
\label{eq:Phi-t-ray}
\Phi_t(X)=\Tr(X)I_d-tX,
\qquad
C_{\Phi_t}=I_d\otimes I_d-td\,P_\omega.
\end{equation}
Write $\alpha=k+\theta$ with $k=\lfloor\alpha\rfloor$ and $\theta\in[0,1)$ and set
\begin{equation}
\label{eq:tstar-def}
t_\alpha^\ast\;:=\;\frac{k+\theta^2}{(k+\theta)^2}.
\end{equation}
By Theorem~\ref{thm:alpha-threshold-Phi-t}, one has $\Phi_t\in\mathsf{P}_\alpha\iff t\le t_\alpha^\ast$.
Combining this with \eqref{eq:fd-explicit} gives an exact reciprocity on the slice.

\begin{corollary}
\label{cor:reciprocity}
For every $\alpha\in[1,d]$,
\begin{equation}
\label{eq:reciprocity}
f_d(\alpha)=\frac{1}{d\,t_\alpha^\ast}.
\end{equation}
In particular, along the ray $\{\Phi_t\}_{t\in(0,1]}$, the boundary between membership and non-membership in $\mathsf{P}_\alpha$
occurs at $t=1/(d\,f_d(\alpha))$.
\end{corollary}

\begin{corollary}
\label{cor:tau-Phi-t-correct}
Fix $d\ge 2$ and let $\Phi_t:\mathbb{M}_d\to\mathbb{M}_d$ be the map
\[
\Phi_t(X)=\Tr(X)I_d-tX.
\]
Then the stability index $\tau(\Phi_t)$ is given by:
\begin{enumerate}
\item[\textup{(i)}] If $t\le 0$, then $\Phi_t$ is completely positive and $\tau(\Phi_t)=d$.
\item[\textup{(ii)}] If $0<t\le \frac{1}{d}$, then $\Phi_t$ is completely positive and $\tau(\Phi_t)=d$.
\item[\textup{(iii)}] If $\frac{1}{d}<t<1$, then $\tau(\Phi_t)$ is the unique $\alpha\in(1,d)$ such that
\begin{equation}
\label{eq:tau-invert}
t=t_\alpha^\ast,
\qquad\text{where}\qquad
t_\alpha^\ast=\frac{k+\theta^2}{(k+\theta)^2}
\ \ \text{for }\ \alpha=k+\theta,\ k=\lfloor\alpha\rfloor,\ \theta\in[0,1).
\end{equation}
More explicitly, if $t\in\big(\frac{1}{k+1},\frac{1}{k}\big]$ for some $k\in\{1,\dots,d-1\}$, then
\[
\tau(\Phi_t)=k+\theta(t),
\]
where $\theta(t)\in[0,1)$ is the unique solution of
\begin{equation}
\label{eq:theta-inverse-t-correct}
(t-1)\theta^2+2tk\,\theta+(tk^2-k)=0.
\end{equation}
Equivalently, for $t\in\big(\frac{1}{k+1},\frac{1}{k}\big)$ one may write
\begin{equation}
\label{eq:theta-explicit-t-correct}
\theta(t)=\frac{tk-\sqrt{k\big(t(k+1)-1\big)}}{1-t}.
\end{equation}
At the endpoint $t=\frac{1}{k}$ one has $\theta(t)=0$ and hence $\tau(\Phi_{1/k})=k$.
\item[\textup{(iv)}] If $t=1$, then $\tau(\Phi_1)=1$.
\end{enumerate}
\end{corollary}

\begin{proof}
Recall from Theorem~\ref{thm:alpha-threshold-Phi-t} that for $t\ge 0$ and $\alpha\in[1,d]$,
\[
\Phi_t\in\mathsf{P}_\alpha
\quad\Longleftrightarrow\quad
t\le t_\alpha^\ast,
\qquad
t_\alpha^\ast=\frac{k+\theta^2}{(k+\theta)^2}
\ \text{ for }\alpha=k+\theta.
\]
Also, for $t\le 0$ the Choi matrix satisfies
$C_{\Phi_t}=I\otimes I+|t|d\,P_\omega\geq 0$, hence $\Phi_t$ is completely positive.

\smallskip
\noindent\emph{Step 1: the completely positive regime.}
For $t\le 1/d$ the Choi matrix
$C_{\Phi_t}=I\otimes I-td\,P_\omega$ is positive semidefinite, so $\Phi_t$ is completely positive.
In particular, $\Phi_t\in\mathsf{P}_\alpha$ for every $\alpha\in[1,d]$, hence
\[
\tau(\Phi_t)=\sup\{\alpha\in[1,d]:\Phi_t\in\mathsf{P}_\alpha\}=d,
\]
proving \textup{(i)}--\textup{(ii)}.

\smallskip
\noindent\emph{Step 2: the intermediate regime $1/d<t<1$.}
In this regime $\Phi_t$ is positive but not completely positive.
Since $\alpha\mapsto t_\alpha^\ast$ is continuous and strictly decreasing on $[1,d]$ with
$t_1^\ast=1$ and $t_d^\ast=1/d$, there exists a unique $\alpha\in(1,d)$ such that $t=t_\alpha^\ast$.
Moreover, by the threshold characterization,
\[
\Phi_t\in\mathsf{P}_\beta \iff t\le t_\beta^\ast,
\]
so the set $\{\beta\in[1,d]:\Phi_t\in\mathsf{P}_\beta\}$ is exactly $[1,\alpha]$, and therefore
$\tau(\Phi_t)=\alpha$. This proves \eqref{eq:tau-invert}.

Now suppose $t\in(\frac{1}{k+1},\frac{1}{k}]$ with $k\in\{1,\dots,d-1\}$ and write $\alpha=k+\theta$
with $\theta\in[0,1)$. The equation $t=t_\alpha^\ast$ becomes
\[
t(k+\theta)^2=k+\theta^2,
\]
which expands to \eqref{eq:theta-inverse-t-correct}. Solving the quadratic yields \eqref{eq:theta-explicit-t-correct},
and the chosen branch is the unique one lying in $[0,1)$ on this interval. At $t=1/k$ the equation forces $\theta=0$,
so $\tau(\Phi_{1/k})=k$.

\smallskip
\noindent\emph{Step 3: the endpoint $t=1$.}
Here $\Phi_1$ is positive but not $1+\varepsilon$-positive for any $\varepsilon>0$
(by Theorem~\ref{thm:alpha-threshold-Phi-t}, since $t_{1+\varepsilon}^\ast<1$),
so $\tau(\Phi_1)=1$.
\end{proof}

\begin{remark}
Since $t_{1}^\ast=1$ and $t_{1+\theta}^\ast<1$ for every $\theta>0$, Theorem~\ref{thm:alpha-threshold-Phi-t} implies
$\tau(\Phi_1)=1$ while $\tau(\Phi_t)>1$ for all $t<1$.
Thus $\tau$ detects, on this canonical ray, the distinction between the endpoint $\Phi_1$ (which does not penetrate beyond the
first fractional level) and every interior point $t<1$ (which lies in $\mathsf{P}_{1+\varepsilon}$ for some $\varepsilon>0$).
\end{remark}


\noindent\textit{Data Availability Statement.} Data sharing is not applicable to this article as no datasets were generated or analyzed during the current study.

\medskip

\noindent \textit{Conflict of Interest Statement.}  There is no conflict of interest.
\medskip

\noindent \textit{Ethical Statement.}  Not applicable. This research did not involve human participants, personal data, or animals.

\end{document}